\documentclass[]{interact}

\usepackage{graphicx}%
\usepackage{subfigure}
\usepackage{multirow}%
\usepackage{amsmath,amssymb,amsfonts}%
\usepackage{amsthm}%
\usepackage{mathrsfs}%
\usepackage[title]{appendix}%
\usepackage{xcolor}%
\usepackage{textcomp}%
\usepackage{manyfoot}%
\usepackage{booktabs}%
\usepackage{algorithm}%
\usepackage{algorithmicx}%
\usepackage{algpseudocode}%
\usepackage{listings}%
\usepackage{CJK}
\usepackage{indentfirst}
\usepackage{cases}
\usepackage{longtable}
\usepackage{array}
\usepackage{caption}
\usepackage{epstopdf}
\usepackage[caption=false]{subfig}

\usepackage[numbers,sort&compress]{natbib}
\bibpunct[, ]{[}{]}{,}{n}{,}{,}
\makeatletter
\def\NAT@def@citea{\def\@citea{\NAT@separator}}
\makeatother

\theoremstyle{plain}
\newtheorem{theorem}{Theorem}[section]
\newtheorem{lemma}[theorem]{Lemma}

\newtheorem{proposition}[theorem]{Proposition}

\theoremstyle{definition}

\newtheorem{assumption}[theorem]{Assumption}

\theoremstyle{remark}

\begin{document}
	
	
	\title{A Fast Smoothing Newton Method for Bilevel Hyperparameter Optimization for SVC with Logistic Loss}
	
	\author{
		\name{Yixin Wang\textsuperscript{a}\thanks{Email: eathelynwang@outlook.com} and 
			Qingna Li\textsuperscript{a,b}\thanks{Email: qnl@bit.edu.cn. Corresponding author. This author's research is supported by NSFC 12071032.}}
		\affil{\textsuperscript{a}School of Mathematics and Statistics, Beijing Institute of Technology, Beijing, China; \textsuperscript{b}Beijing Key Laboratory on MCAACI/Key Laboratory of Mathematical Theory and Computation in Information Security, Beijing Institute of Technology, Beijing, China}
	}
	
	\maketitle
	
	\begin{abstract}
		Support vector classification (SVC) with logistic loss has excellent theoretical properties in classification problems where the label values are not continuous. 
		In this paper, we reformulate the hyperparameter selection for SVC with logistic loss as a bilevel optimization problem in which the upper-level problem and the lower-level problem are both based on logistic loss.
		The resulting bilevel optimization model is converted to a single-level nonlinear programming (NLP) based on the KKT conditions of the lower-level problem. Such NLP contains a set of nonlinear equality constraints and a simple lower bound constraint. The second-order sufficient condition is characterized, which guarantees that the strict local optimizers are obtained.
		To solve such NLP, we apply the smoothing Newton method proposed in \cite{Liang} to solve the KKT conditions, which contain one pair of complementarity constraints. We show that the smoothing Newton method has a superlinear convergence rate. 
		Extensive numerical results verify the efficiency of the proposed approach and strict local minimizers can be achieved both numerically and theoretically. In particular, compared with other methods, our algorithm can achieve competitive results while consuming less time than other methods.
	\end{abstract}
	
	\begin{keywords}
		Support vector classification, Hyperparameter selection, Bilevel optimization, Logistic loss, Smoothing Newton Method, Superlinear Convergence
	\end{keywords}

	\section{Introduction}\label{sec1}
	Support vector classification (SVC) is an optimal margin based classification technique, which is a classical and well-performed learning method for classification problems in machine learning \cite{Chauhan, Chapelle,Cortes,Duan,Keerthi,Vapnik}. 
	In those models, hyperparameter selection is a critical issue and has been addressed by many researchers both theoretically and practically \cite{Dong,Kunapuli,Kunapuli_a,Kunapuli_b}. Cross validation is usually conducted in order to choose hyperparameters.
	A classical approach for cross-validation is the grid search method \cite{Momma}, where one needs to define a hyperparameter grid, and search for the optimal hyperparameters to minimize the cross-validation error (CV error). 
	Bennett et al. \cite{Bennett} emphasize that one of the drawbacks of the grid search approach is that the continuity of the hyperparameter is ignored by the discretization.
	
	In terms of selecting hyperparameters through bilevel optimization, which is a very active field of applied mathematics, different models and approaches have been considered. The main reason is that bilevel optimization problems can serve as a powerful tool for modeling hierarchical decision making processes \cite{LiQ}. This cross-validation optimization problem can be formulated as a bilevel program which is significantly more versatile than the commonly used grid search procedure \cite{Kunisch}.
	Colson et al. \cite{Colson} proceed towards a general formulation of  bilevel optimization and then focus on solution approaches with mathematical program with equilibrium constraints (MPEC).
	Moore et al. \cite{Moore} propose a nonsmooth bilevel model and design a proximity control approximation algorithm to select hyperparameters with T-fold cross-validation. 
	Okuno et al. \cite{Okuno} propose a bilevel optimization model to select the best hyperparameter for a nonsmooth, possibly non-convex, $l_p$-regularized problem. They present a smoothing-type algorithm with convergence analysis to solve this bilevel optimization model. 
	Li et al. \cite{LiQ} establish a bilevel programming model for support vector classification, transforming it into an MPEC, and using the global relaxing method (GRM) to solve it, which improved the accuracy of classification prediction. In particular, they show that due to the structure of the problem, the Mangasarian–Fromovitz constraint qualification (MFCQ) holds for such MPEC, which guarantees the convergence of GRM to a C-stationary point. In \cite{LiZ}, the authors further investigate the same framework with lower-level loss function replaced by $l_2$-loss. Interestingly, MPEC-MFCQ also holds for the resulting MPEC. 
	In \cite{LiQ_b}, the authors proposed an efficient LP-Newton global relaxation method to tackle the MPEC derived from the bilevel model in \cite{Kunapuli_a}, which involves multiple hyperparameters to deal with feature selection.
	In \cite{Coniglio}, the authors further characterized the linearly independent constraint qualification for MPEC (LICQ-MPEC) derived from the bilevel model for hyperparameter selection in kernel SVC. 
	
	Besides the $l_1$-loss SVC and the $l_2$-loss SVC, the logistic-loss SVC is also widely used. The logistic-loss SVC was originally proposed by Nelder and Wedderburn in 1972 \cite{Nelder}. They show that the logistic-loss function constructed based on entropy and conditional probability has better theoretical properties and it uses the maximum likelihood estimation of the probability log function to fit the parameters with the goal of maximizing the probability sense.  
	
	Logistic loss has shown strong superiority in classification prediction \cite{Dayton,Franceschini,Glonek,LaValley,Salehi,Tripepi}. Compared with $l_1$-loss function and $l_2$-loss function, the advantage of logistic-loss function lies in two aspects.
	Firstly, the logistic-loss function is smooth, whereas the $l_1$-loss function and $l_2$-loss function are not. By introducing extra variables and inequalities, SVC with $l_1$-loss and $l_2$-loss functions can be transferred into smooth optimization problems, which lead to complementarity constraints in the resulting single-level optimization problem. However, it increases the scale of variables and constraints, and the resulting MPEC is much more difficult to solve. In contract, as we will see later, the logistic-loss SVC in lower-level problem will lead to a nice equality constraint in the resulting single level NLP.
	Secondly, the $l_1$-loss and $l_2$-loss functions are susceptible to extreme outliers, and small errors near the boundary $\left({x_i}^\top w=0\right)$ are easy to be ignored, which reduces the accuracy of classification. However, the logistic loss function is sensitive near the boundary $\left({x_i}^\top w=0\right)$ and it reduces the effect of outliers \cite{Kleinbaum}. Therefore, the logistic loss can make the model pay more attention to the classification boundary: errors near the boundary line are strengthened, which are less likely to be ignored in the optimization process, and the influence of extreme outliers is weakened, thereby increasing the stability of the model \cite{Nguyen}.
	
	Due to the above reasons, logistic-loss SVC is extensively studied and is applied to a wide range of biological and socio-technical systems \cite{Abdalrada,Stoltzfus,Meyer,Walkling}. 
	Pregibon \cite{Pregibon} develops diagnostic measures to aid the analyst in detectingsuch observations and in quantifying their effect on various aspects of the maximum likelihood fit of a logistic model.
	Peduzzi et al. \cite{Peduzzi} perform a Monte Carlo study to evaluate the effect of the number of events per variable (EPV) based on logistic regression analysis.	
	Chaudhuri et al. \cite{Chaudhuri} use regularized logistic loss to establish a privacy protection mechanism, which balances data privacy protection and classification learning performance to a certain extent.
	Le et al. \cite{Le} proposed a fast and accurate approach that utilises GPU-based extreme gradient boosting machine using squared logistic loss (SqLL) for bankruptcy forecasting. 
	
	To summarize, a natural question is that what will happen to the bilevel model for SVC with logistic loss. This motivate the work in our paper. In this paper, we will study the bilevel optimization model based on logistic loss for selecting the hyperparameter in logistic-loss SVC. 
	
	The contributions of the paper are as follows. 
	Firstly, we propose a bilevel optimization model based on logistic functions for hyperparameter selection in a binary SVC. 
	Secondly, we convert the bilevel optimization model to an NLP, and investigate the first and second order optimality conditions. 
	Thirdly, we apply the squared smoothing Newton method studied in \cite{Liang} to solve the KKT system of NLP. A superlinear convergence rate is achieved for the proposed method. 
	Finally, we conduct extensive numerical experiments to verify the efficiency of the proposed method.
	
	The paper is organized as follows. 
	In Sect. 2, based on T-fold cross-validation, we introduce a bilevel optimization model based on logistic functions to select an optimal hyperparameter for SVC.
	In Sect. 3, we transfer the bilevel optimization model to an NLP and analyze the interesting properties including critical cone and second-order sufficient conditions. 
	In Sect. 4, we smooth the KKT conditions of NLP by Huber function, and reformulate the smoothing system as a nonlinear equation system. We apply a squared smoothing Newton method, and prove that the Jacobian matrix of the nonlinear equation system is nonsingular, which is a key property to guarantee the superlinearly convergence of the algorithm. 
	In Sect. 5, we conduct extensive numerical experiments, and compare the results with other methods. 
	Finally, we conclude the paper in Sect. 6.
	
	$\mathbf{Notations.}$
	We use $[k]$ to denote $\{1,2,...,k\}$ for the positive integer $k$.
	We use ':=' to represent 'define'.
	For $v\in\mathbb{R}^m$, we use $e^{-v}$ ($\log{v}$) to denote a column vector in $\mathbb{R}^m$ with each component $e^{-v_i}\ (\log{v_i}),\ i\in[m]$.
	We use notations in bold to represent a constant matrix or vector with proper size. For example, $\mathbf{1}_{m\times m}$ represents an $m$ by $m$ matrix with all elements one, and $\mathbf{0}_{m}$ represents a column vector of size $m$ with all elements zeros.

	\section{Bilevel Model for Hyperparameter Selection in Logistic-SVC}\label{sec2}
	In this part, we establish the bilevel model for hyperparameter selection in logistic-loss SVC based on cross validation.
	
	In order to choose the hyperparameter in logistic-loss SVC, the $T$-fold cross validation works in the following way. First we split the data set evenly into $T$ parts \cite{LiZ}, each part denoted by
	$\left\lbrace \left( \bar{x}_i^{\left( j\right)},\bar{y}_i^{\left( j\right)}\right)\right\rbrace _{i=1}^{m_1},\ j\in [T],$
	where $m_1$ denotes the size of the $j$-th part. $\bar{y}_i^{(j)}\in\{\pm 1\}$ is the label corresponding to data $\bar{x}_i^{(j)}\in\mathbb{R}^n,i\in [m_1]$. For each fold $j\in[T]$, we take the $j$-th part as the validation set, and the rest parts as the training set (denote as $\left\lbrace \left( \hat{x}_i^{\left( j\right)},\hat{y}_i^{\left( j\right)}\right)\right\rbrace _{i=1}^{m_2}$) to train the model. The best parameter is chosen so that the total loss on each validation set is minimized. 
	
	To be specified, we start with the lower-level problem (take the $j$-th lower-level problem as an example), which is to search for a hyperplane ${w^{\left( j\right)}}^\top x=0$ such that the following logistic-loss SVC model is optimized (given $C>0$)
	\begin{equation} \label{eq_lower_ori}
		\min\limits_{w^{(j)}\in\mathbb{R}^n}\frac{1}{2}\left\|w^{\left(j\right)}\right\|_2^2+C\sum\limits_{i=1}^{m_2} J_{w^{\left(j\right)}} \left( \hat{x}_i^{\left( j\right)},\hat{y}_i^{\left( j\right)}\right):=F\left(w^{\left( j\right) }\right),\ j\in[T]
	\end{equation}
	where $J_{w^{\left(j\right)}}\left( \hat{x}_i^{\left( j\right)},\hat{y}_i^{\left( j\right)}\right)$ is the logistic-loss function (for $\hat{y}_i\in\left\lbrace \pm 1\right\rbrace $) defined by 
	\begin{equation*}
		J_{w^{\left(j\right)}}\left( \hat{x}_i^{\left( j\right)},\hat{y}_i^{\left( j\right)}\right)
		=\frac{1}{2}\left(1+\hat{y}_i^{(j)}\right)\log{\left(1+e^{-{\hat{x}_i}^{{(j)}^\top} w^{(j)}}\right)}
		+\frac{1}{2}\left(1-\hat{y}_i^{\left(j\right)}\right)\log{\left(1+e^{{\hat{x}_i}^{{(j)}^\top} w^{(j)}}\right)}.
	\end{equation*}
	
	In terms of the upper-level problem, the aim is to minimize the loss function over the validation set for all the $T$ folds. Here, we also use the logistic loss as our upper-level loss function \cite{LiT}, which is
	$\log{\left(1+e^{-\bar{y}_i^{\left(j\right)}\left({\bar{x}_i} ^ {{\left(j\right)}^\top} w^{\left(j\right)}\right)}\right)},\ j\in[T].$
	
	Overall, we reach the following bilevel model to choose the best parameter $C$
	\begin{equation}\label{eq_bilevel_ori}
		\begin{aligned}
			\min\limits_{\substack{C\in\mathbb{R},w^{\left(j\right)}\in\mathbb{R}^n,j\in[T]}} & \frac{1}{Tm_1}{\sum\limits_{t=1}^{T} {\sum\limits_{i=1}^{m_1} { \log{\left(1+e^{-\bar{y}_i^{\left(j\right)}\left({\bar{x}_i} ^ {{\left(j\right)}^\top} w^{\left(j\right)}\right)}\right)} }}}\\
			{\rm s.t.\quad} &C\ge0,\\
			&{\rm and\ for\ } j=1,2,...,T,\\
			&\qquad\  w^{\left(j\right)}=\arg\min\limits_{w^{\left(j\right)}\in\mathbb{R}^n}\left\lbrace \frac{1}{2}\left\|w^{\left(j\right)}\right\|_2^2+C\sum\limits_{i=1}^{m_2} J_{w^{\left(j\right)}}\left( \hat{x}_i^{\left( j\right)},\hat{y}_i^{\left( j\right)}\right)\right\rbrace .
		\end{aligned}
	\end{equation}
	
	\section{Single Level Reformulation and the Theoretical Properties}\label{sec3}
	In this section, we first reformulate the bilevel optimization problem as a single-level optimization problem based on the KKT conditions of the lower-level problems. Then we present some theoretical properties of the single-level problem, including the KKT conditions and the second-order sufficient conditions. 

	\subsection{Single-level reformulation}\label{subsec3-1}
	Note that each lower-level problem is an unconstrained optimization problem with smooth objective function $F\left(w^{(j)}\right)$. Due to the convexity of the logistic-loss function, the objective function in the lower-level problem is strongly convex. The solution is unique for each lower-level problem (given $C>0$). A natural way to represent the lower-level problem is by its first order optimality condition. That is,
	\begin{center}
		$\nabla F\left(w^{\left(j\right)}\right)=0,\ j\in [T],$
	\end{center}
	which is
	\begin{equation}\label{eq_kkt_1}
		w^{\left(j\right)}+\frac{1}{2}C\sum_{i=1}^{m_2}\left(h_{\hat{x}_i^{\left(j\right)}}\left(w^{(j)}\right)-\hat{y}_i^{\left(j\right)}\right)\hat{x}_i^{\left(j\right)}=\mathbf{0},\ 
		h_{\hat{x}_i^{\left(j\right)}}\left(w^{(j)}\right):=\frac{1-e^{-\hat{x}_i^{{\left(j\right)}^\top} w^{(j)}}}{1+e^{-\hat{x}_i^{{\left(j\right)}^\top} w^{(j)}}},\ j\in [T].
	\end{equation}
	
	By replacing each lower-level problem by (\ref{eq_kkt_1}), the bilevel problem (\ref{eq_bilevel_ori}) can be equivalently written into the following single-level optimization problem.
	\begin{equation}\label{eq_NLP}
		\begin{aligned}
			\min\limits_{\substack{C\in\mathbb{R},w^{\left(j\right)}\in\mathbb{R}^n,j\in[T]}} & \frac{1}{Tm_1}{\sum\limits_{t=1}^{T} {\sum\limits_{i=1}^{m_1} { \log{\left(1+e^{-\bar{y}_i^{\left(j\right)}\left({\bar{x}_i} ^ {{\left(j\right)}^\top} w^{\left(j\right)}\right)}\right)} }}}\\
			{\rm s.t.\quad} &C\ge0,\\
			&w^{\left(j\right)}+\frac{1}{2}C\sum_{i=1}^{m_2}\left(\hat{x}_i^{\left(j\right)}\left(h_{\hat{x}_i^{\left(j\right)}}\left(w^{(j)}\right)-\hat{y}_i^{\left(j\right)}\right)\right)=\mathbf{0},\ j\in[T].
		\end{aligned}
	\end{equation}
	
	For (\ref{eq_NLP}), it is a typical NLP with one simple lower bound constraint and a set of nonlinear equality constraints. Compared with the MPEC reformulation of the bilevel model for hyperparameter selection in $l_1$-SVC \cite{LiQ} and $l_2$-SVC \cite{LiZ}, the above resulting NLP demonstrates the advantages in the following sense. Firstly, there is no complementarity constraints in (\ref{eq_NLP}). Secondly, the objective function and the constraints in NLP are both smooth. Thirdly, the number of the equality constraints in (\ref{eq_NLP}) is $Tn$, which is much smaller than that in MPEC reformulation in \cite{LiQ,LiZ}. Finally, to solve (\ref{eq_NLP}), lots of efficient numerical algorithms for NLP can be applied, whereas for solving MPEC in \cite{LiQ,LiZ}, efficient yet fast algorithms are still highly in need.
	
	(\ref{eq_NLP}) can be equivalently written in the matrix form as follows
	\begin{equation}\label{eq_single_mat}
		\begin{aligned}
			\min\limits_{\substack{C\in\mathbb{R},w\in\mathbb{R}^{Tn}}}
			&\frac{1}{Tm_1}\mathbf{1}^\top\log{\left(\mathbf{1}+e^{-Aw}\right)}\\
			{\rm s.t.\quad} & \begin{cases} 
				C\geq0, \\
				w+\frac{1}{2}CX\left(h_X\left(w\right)-\hat{y}\right)=\mathbf{0}.
			\end{cases}
		\end{aligned}
	\end{equation}		
	Here $w\in \mathbb{R}^{Tn},\ \hat{y}\in\mathbb{R}^{Tm_2},\ X\in \mathbb{R}^{Tn\times Tm_2},\ A\in\mathbb{R}^{Tm_1\times Tn},\ h_X(w):\mathbb{R}^{Tn} \rightarrow \mathbb{R}^{Tm_2}$ are defined by
	\begin{equation*}
		w:=\left[ \begin{array}{c}
			w^{\left(1\right)}\\ \vdots \\w^{\left(T\right)}
		\end{array} \right],\ 
		\hat{y}:=\left[ \begin{array}{c}
			\hat{y}^{\left(1\right)}\\\vdots\\ \hat{y}^{\left(T\right)}
		\end{array} \right],\ 
		X:=\begin{bmatrix} 
			X^{\left(1\right)} & \cdots & \mathbf{0} \\
			\vdots &\ddots & \vdots \\
			\mathbf{0} & \cdots & X^{\left(T\right)}
		\end{bmatrix},\ 
		A:=\begin{bmatrix} 
			A^{\left(1\right)}& \cdots & \mathbf{0} \\
			\vdots & \ddots & \vdots \\
			\mathbf{0} & \cdots & A^{\left(T\right)}
		\end{bmatrix},
	\end{equation*}
	\begin{equation*}		
		h_X\left(w\right):=\left[ \begin{array}{c}
			h_{X^{(1)}}(w^{(1)})\\ \vdots \\ h_{X^{(T)}}(w^{(T)})
		\end{array} \right],\ 
		\hat{y}^{\left(j\right)}:=\left[ \begin{array}{c}
			\hat{y}_1^{\left(j\right)}\\ \vdots\\
			\hat{y}_{m_2}^{\left(j\right)}
		\end{array} \right]\in\mathbb{R}^{m_2},\ 
		X^{\left(j\right)}: = \left[\hat{x}_1^{\left(j\right)},\cdots,\hat{x}_{m_2}^{\left(j\right)}\right]\in\mathbb{R}^{n\times m_2},
	\end{equation*}
	\begin{equation*}		
		A^{\left(j\right)}:=\left[ \begin{array}{c}
			A_1^{(j)} \\ \vdots \\ A_{m_1}^{(j)}
		\end{array} \right]\in\mathbb{R}^{m_1\times n},\ 
		A_i^{(j)}:=\left(\bar{y}_i^{(j)}\bar{x}_i^{\left(j\right)}\right)^\top,\ 
		h_{X^{(j)}}(w^{(j)}):=\left[ \begin{array}{c}
			h_{{\hat{x}_1}^{(j)}}(w^{(j)})\\
			h_{{\hat{x}_2}^{(j)}}(w^{(j)})\\ \vdots \\
			h_{{\hat{x}_{m_2}}^{(j)}}(w^{(j)})
		\end{array} \right]\in\mathbb{R}^{m_2},\ j\in [T].
	\end{equation*}
	
	\subsection{KKT conditions for (\ref{eq_single_mat})}\label{subsec3-2}
	The Lagrange function of problem (\ref{eq_single_mat}) is as follows:
	\begin{equation}\label{eq_defL}
		\begin{aligned}
			L\left(C,w;\lambda,\mu\right)&=\frac{1}{Tm_1}\mathbf{1}^\top\log{\left(\mathbf{1}+e^{-Aw}\right)}-\lambda C+\mu^\top\left(w+\frac{1}{2}CX\left(h_X\left(w\right)-\hat{y}\right)\right)\\
			&=\frac{1}{Tm_1}f\left(w\right)-\lambda C+\mu^\top\left(w+Cg\left(w\right)\right),
		\end{aligned}
	\end{equation}
	where $\mu \in \mathbb{R}^{Tn},\ \lambda\in\mathbb{R}$ are the Lagrange multipliers corresponding to equality and inequality constraints, respectively, and
	\begin{equation*}
		f(w):= \mathbf{1}^\top\log {\left(\mathbf{1}+e^{-Aw}\right)},\ 
		g(w):= \frac{1}{2}X\left(h_X\left(w\right)-\hat{y}\right).
	\end{equation*}
	
	The KKT conditions of (\ref{eq_single_mat}) are given by 
	\begin{subnumcases}{\label{eq_kkt_ori}}
		\nabla_C L(C,w;\lambda,\mu)=-\lambda+\mu^\top g\left(w\right)=0,\label{eq_kkt_ori_a}\\
		\nabla_w L(C,w;\lambda,\mu)=\frac{1}{Tm_1}\nabla f(w)+\left(I_{Tn}+C\nabla g(w)\right)\mu = \mathbf{0},\label{eq_kkt_ori_b}\\
		0\le\lambda\bot C\geq0,\label{eq_kkt_ori_c}\\
		w+Cg\left(w\right)=\mathbf{0},\label{eq_kkt_ori_d}
	\end{subnumcases}
	with 
	\begin{equation}\label{eq_df}
		\nabla f(w)=-A^\top s(w)\in\mathbb{R}^{Tn},\ 
		\nabla g(w)=X{\rm Diag}(u(w))X^\top\in\mathbb{R}^{Tn\times Tn},
	\end{equation} 
	\begin{equation}\label{eq_sw}
		s\left(w\right)=\left[ \begin{array}{c}
			s^{(1)}(w^{(1)})\\ \vdots \\ s^{(T)}(w^{(T)})
		\end{array} \right]\in\mathbb{R}^{Tm_1},\ 
		s^{(j)}(w^{(j)})=\left[ \begin{array}{c}
			{\left(1+e^{A_1^{\left(j\right)}w^{\left(j\right)}}\right)}^{-1}\\
			\vdots \\
			{\left(1+e^{A_{m_1}^{\left(j\right)}w^{\left(j\right)}}\right)}^{-1}
		\end{array} \right]\in\mathbb{R}^{m_1},\ j\in [T],
	\end{equation}
	\begin{equation}\label{eq_uw}
		u\left(w\right)=\left[ \begin{array}{c}
			u^{(1)}(w^{(1)})\\ \vdots \\ u^{(T)}(w^{(T)})
		\end{array} \right]\in\mathbb{R}^{Tm_2},\ 
		u^{(j)}(w^{(j)})=\left[ \begin{array}{c}
			{\left(2+e^{-{\hat{x}_1^{(j)^\top} w^{(j)}}}+e^{\hat{x}_1^{(j)^\top} w^{(j)}}\right)}^{-1}\\ \vdots \\
			{\left(2+e^{-{\hat{x}_{m_2}^{(j)^\top} w^{(j)}}}+e^{\hat{x}_{m_2}^{(j)^\top} w^{(j)}}\right)}^{-1}\\
		\end{array} \right]\in\mathbb{R}^{m_2}.
	\end{equation}
	\vskip 1cm
	Let $v=[C;w]\in\mathbb{R}^{Tn+1}$. Denote the active index set for inequality constraint $C>0$ as follows
	\begin{equation}\label{eq_I}
		\mathcal{I}(v)=\{1\}\ {\rm if}\ C=0\ {\rm and}\ \mathcal{I}(v)=\emptyset\ {\rm if}\ C>0.		
	\end{equation}
	
	Below we show that for each feasible point $v\in\mathbb{R}^{Tn+1}$ of (\ref{eq_single_mat}), the linearly independent constraint qualification (LICQ) holds.
	\begin{proposition}
		LICQ holds at each feasible point $v\in\mathbb{R}^{Tn+1}$ of (\ref{eq_single_mat}).
	\end{proposition}
	\begin{proof}
		For simplicity, denote
		\begin{equation*}
			\theta(v):=w+\frac{1}{2}CX\left(h_X\left(w\right)-\hat{y}\right)=
			\left[\begin{array}{c}
				\theta_1(v) \\ \vdots \\ \theta_{Tn}(v)
			\end{array}\right]\in\mathbb{R}^{Tn}.
		\end{equation*}
		LICQ holds at $v$ if and only if the following vectors 
		\begin{equation}\label{eq_licq}
			\left\lbrace \nabla\theta_i(v)|\ i\in[Tn]\right\rbrace \cup \left\lbrace [\ (1;0;0;...;0)\in\mathbb{R}^{Tn+1}]|\ {\rm if}\ \mathcal{I}(v)=\{1\}\right\rbrace 
		\end{equation}
		are linearly independent.
		
		Note that 
		\begin{equation}\label{eq_d_theta}
			\begin{aligned}
				\left[\begin{array}{c}
					\nabla\theta_1(v)^\top \\ \vdots \\ \nabla\theta_{Tn}(v)^\top
				\end{array}\right]
				&=\left[\begin{array}{cc}
					\nabla_C\theta_1(v) & \nabla_w\theta_1(v)^\top\\ \vdots & \vdots\\ \nabla_C\theta_{Tn}(v) & \nabla_w\theta_{Tn}(v)^\top
				\end{array}\right] \\
				&=\left[\begin{array}{cc} g(w) & I_{Tn}+CX{\rm Diag}(u(w))X^\top \end{array}\right]\\
				&:=\left[\begin{array}{cc} g(w) & P(v) \end{array}\right].
			\end{aligned}
		\end{equation}
		By definition of $u(w)$ in (\ref{eq_uw}), all elements in $u(w)$ is positive. It implies that $X{\rm Diag}(u)X^\top$ is positive semidefinite. Therefore, $P(v)$ is positive definite for any $C\ge0$.
		
		Consider the following two cases. If $C=0$, $\mathcal{I}(v)=\{1\}$. The row vectors in (\ref{eq_licq}) reduces to the row vectors in the following matrix
		\begin{equation*}
			\left[\begin{array}{cc}
				g(w) & P(v) \\ 1 & \mathbf{0}_{1\times Tn} 
			\end{array}\right]=
			\left[\begin{array}{cc}
				g(w) & I_{Tn} \\ 1 & \mathbf{0}_{1\times Tn} 
			\end{array}\right]\in\mathbb{R}^{Tn}.
		\end{equation*}
		Obviously, the row vectors in the above matrix are linearly independent, implying LICQ holds.
		If $C>0$, $\mathcal{I}(v)=\emptyset$. The row vectors in (\ref{eq_d_theta}) are linearly independent due to the positive definiteness of $P(v)$. Overall, LICQ holds at any feasible point $v$ of (\ref{eq_single_mat}).
	\end{proof}

	\subsection{Second-order optimality conditions for (\ref{eq_single_mat})}\label{subsec3-3}
	To further analyze the second-order optimality conditions for (\ref{eq_single_mat}), we first need to characterize the critical cone. Assume that $(v^*,\lambda^*,\mu^*)$ satisfy the KKT system (\ref{eq_kkt_ori}).
	
	\begin{proposition}\label{Proposition_cone}
		Let $(v^*,\lambda^*,\mu^*)$ satisfy the KKT conditions (\ref{eq_kkt_ori}). The critical cone $\mathcal{C}(v^*,\lambda^*,\mu^*)$ takes the following form:\\
		(i) If $C^*>0$, 
		$\mathcal{C}(v^*,\lambda^*,\mu^*)=\left\lbrace d\in \mathbb{R}^{Tn+1}\ \Bigg|\ d=\left[\begin{array}{c} 1\\-P(v^*)^{-1} g(w^*)\end{array}\right]d^C,\ d^C\in\mathbb{R} \right\rbrace.$\\
		(ii) If $C^*=0$ and $\lambda^*=0$, 
		$\mathcal{C}(v^*;\lambda^*,\mu^*)=\left\lbrace d\ \Bigg|\ d=\left[\begin{array}{c}1\\-g\left(w^*\right)
		\end{array}\right] d^C,\ d^C\ge0 \right\rbrace. $\\
		(iii) If $C^*=0$ and $\lambda^*>0$,
		$\mathcal{C}(v^*;\lambda^*,\mu^*)=\left\lbrace \mathbf{0}_{Tn+1} \right\rbrace$.
	\end{proposition}
	
	\begin{proof}
		The set of linearized feasible directions at $v^*$, denoted by $\mathcal{F}(v^*)$, is given by
		\begin{equation*}
			\begin{aligned}
				\mathcal{F}(v^*)=\{ d=\left[ d^C; d^w\right]\in \mathbb{R}^{Tn+1}\ |\ 
				&d^C\in\mathbb{R},\ d^w\in\mathbb{R}^{Tn},\ g(w^*)d^C+P(v^*)d^w=0,\\
				&d^C\ge0\ {\rm if}\ \mathcal{I}(v^*)\neq\emptyset\}.
			\end{aligned}
		\end{equation*}
		
		The critical cone at $(v^*,\lambda^*,\mu^*)$ is defined as follows
		\begin{equation}\label{eq_cone}
			\begin{aligned}
				\mathcal{C}(v^*,\lambda^*,\mu^*) & =\{d\in\mathcal{F}(v^*)\ |\ d^C=0,\ i\in \mathcal{I}(v^*)\ {\rm with}\ \lambda_i^*>0\}\\
				& =\{ d=\left[d^C; d^w\right]\in \mathbb{R}^{Tn+1}\ |\ d^C\in\mathbb{R},\ d^w\in\mathbb{R}^{Tn},\ g(w^*)d^C+P(v^*) d^w=0,\\
				& \qquad\qquad\qquad\qquad\qquad\qquad\  d^C=0 \ {\rm if}\ C^*=0 \ {\rm and}\ \lambda^*>0,\\		
				& \qquad\qquad\qquad\qquad\qquad\qquad\  d^C\ge0 \ {\rm if}\ C^*=0 \ {\rm and}\ \lambda^*=0\}.
			\end{aligned}
		\end{equation}
		
		Below we discuss three situations.
		
		Case 1. $C^*>0$. In this case, $\mathcal{I}(v^*)=\{1\}$. By (\ref{eq_cone}), $\mathcal{C}(v^*,\lambda^*,\mu^*)$ reduces to the following form:
		\begin{equation*}
			C(v^*,\lambda^*,\mu^*)=\{d=\left[d^C; d^w\right]\in \mathbb{R}^{Tn+1}\ |\ d^C\in\mathbb{R},\ d^w\in\mathbb{R}^{Tn},\ g(w^*)d^C+P(v^*) d^w=0\}.
		\end{equation*}
		Note that $P(v^*)$ is positive definite for $C\ge0$. We obtain that $d^w=-P(v^*)^{-1}g(w^*)d^C$. In this case,
		\begin{equation*}
			\mathcal{C}(v^*,\lambda^*,\mu^*)=\left\lbrace \left[\begin{array}{c}1\\-P(v^*)^{-1}g(w^*)\end{array}\right]d^C\ \Bigg|\ d^C\in\mathbb{R}
			\right\rbrace.
		\end{equation*}
		
		Case 2. $C^*=0$ and $\lambda^*=0$. In this case, (\ref{eq_cone}) reduces to the following 
		\begin{equation}\label{c2}
			\mathcal{C}(v^*,\lambda^*,\mu^*) = \left\lbrace d=\left[d^C;d^w\right]\in \mathbb{R}^{Tn+1} \Big| g(w^*)d^C+d^w=0,\ d^C\ge0,\ d^w\in\mathbb{R}^{Tn} \right\rbrace.
		\end{equation}
		One can obtain that $d^w=-g(w)d^C$, which gives that $d\in \mathcal{C}(v^*,\lambda^*,\mu^*)$ takes the following form
		\begin{equation*}
			d=\left[ \begin{array}{c} 1 \\ -g(w^*) \end{array} \right]d^C,\ d^C\ge0.
		\end{equation*}
		
		Case 3. $C^*=0$ and $\lambda^*>0$. (\ref{eq_cone}) reduces to the following form
		\begin{equation*}	
			\mathcal{C}(v^*,\lambda^*,\mu^*) = \left\lbrace d\in \mathbb{R}^{Tn+1} \Bigg| \left[ \begin{array}{cc} 1 & \mathbf{0}_{1\times Tn} \\ g(w^*) & I_{Tn} \end{array} \right]   d = \mathbf{0} \right\rbrace.
		\end{equation*}
		Since $\left[ \begin{array}{cc} 1 & \mathbf{0}_{1\times Tn} \\ g(w^*) & I_{Tn} \end{array} \right]$ is nonsingular, $\mathcal{C}(v^*,\lambda^*,\mu^*)=\left\lbrace \mathbf{0}_{Tn+1}\right\rbrace $. The proof is finished.	
	\end{proof}
	
	Define $\alpha(\mu,v)\in\mathbb{R}^{Tn\times Tn}$,  $z(\mu,v)\in\mathbb{R}$ and $\iota(\mu)\in\mathbb{R}$ by
	\begin{subnumcases}{}
		\alpha(\mu,v):=\frac{1}{Tm_1}A^\top {\rm Diag}(p(w))A+CX{\rm Diag}(X^\top\mu){\rm Diag}(q(w))X^\top
		\label{eq_alpha}\\
		z(\mu,v):=g \left(w\right) ^\top P(v)^{-1} \left( \alpha\left(\mu,v\right) P\left(v\right) ^{-1} g\left(w\right) - 2\nabla g\left(w\right) \mu\right),\label{eq_z}\\
		\iota(\mu):=\frac{1}{4}\mu^\top XX^\top X\hat{y} +\frac{1}{16Tm_1}\left(X\hat{y}\right)^\top A^\top AX\hat{y}.	
	\end{subnumcases}
	\begin{assumption}\label{Assumption1}
		Assume that $z(\mu^*,w^*)>0$.
	\end{assumption}
	\begin{assumption}\label{Assumption2}
		Assume that $\iota(\mu^*)>0$.
	\end{assumption}
	\begin{theorem}\label{Theorem_hessian}
		Let $v^*=\left[C^*;w^*\right]\in\mathbb{R}^{1+Tn}$ with the Lagrange multiplier  $\gamma^*=\left[\lambda^*;\mu^*\right]\in\mathbb{R}^{1+Tn}$ satisfy the KKT conditions (\ref{eq_kkt_ori}). 
		If one of the three following conditions holds, then $v^*$ is a strict local minimizer of (\ref{eq_single_mat}):\\
	 	(i) $C^*>0$ and Assumption \ref{Assumption1} hold;\\
		(ii) $C^*=0,\ \lambda^*=0$ and Assumption \ref{Assumption2} hold;\\
		(iii) $C^*=0$ and $\lambda^*>0$.
	\end{theorem}
	\begin{proof}
		First, we derive the formulation of $\nabla_{vv}^2 L\left(v^*;\lambda^*,\mu^*\right)$. By the definition of Lagrange function in (\ref{eq_defL}), it holds that 
		\begin{equation*}
			\nabla_{vv}^2L\left(v;\lambda,\mu\right)
			=\left[\begin{array}{cc}
				\nabla_{CC}^2 L(v;\lambda,\mu) & \nabla_{Cw}^2 L(v;\lambda,\mu) \\ 
				\nabla_{Cw}^2 L(v;\lambda,\mu)^\top & \nabla_{ww}^2 L(v;\lambda,\mu)
			\end{array}\right].
		\end{equation*}
		
		Note that $\nabla_{CC}^2 L(v;\lambda,\mu)=0$, $\nabla_{Cw}^2 L(v;\lambda,\mu)=\mu^\top \nabla g(w)=\mu^\top X{\rm Diag}(u(w))X^\top$.
		
		Below we derive $\nabla_{ww}^2 L(v;\lambda,\mu)$. By (\ref{eq_kkt_ori_b}), it holds that
		\begin{equation*}
			\begin{aligned}
				\nabla_{ww}^2 L(v;\lambda,\mu)
				& = \nabla_w \left(\frac{1}{Tm_1}\nabla f(w)+ C\nabla g(w) \mu\right)\qquad\qquad\quad\ \ ({\rm by\ (\ref{eq_kkt_ori_b})})\\
				& = \nabla _w\left(\frac{-A^\top s(w)}{Tm_1}+ CX{\rm Diag}(u(w))X^\top\mu\right)\qquad({\rm by\ (\ref{eq_df})})\\
				& = \nabla _w\left(\frac{-A^\top s(w)}{Tm_1}+ CX{\rm Diag}(X^\top\mu)u(w)\right)\\
				& := \frac{1}{Tm_1}A^\top {\rm Diag}(p(w))A+CX{\rm Diag}(X^\top\mu){\rm Diag}(q(w))X^\top,
			\end{aligned}
		\end{equation*}
		where
		\begin{equation*}
			p\left(w\right):=\left[ \begin{array}{c}
				p^{(1)}(w^{(1)})\\ \vdots \\ p^{(T)}(w^{(T)})
			\end{array} \right]\in\mathbb{R}^{Tm_1},\ 
			p^{(j)}(w^{(j)}):=\left[ \begin{array}{c}
				\frac{e^{A_1^{\left(j\right)}w^{\left(j\right)}}}{\left(1+e^{A_i^{\left(j\right)}w^{\left(j\right)}}\right)^2}\\ \vdots \\
				\frac{e^{A_{m_1}^{\left(j\right)}w^{\left(j\right)}}}{\left(1+e^{A_i^{\left(j\right)}w^{\left(j\right)}}\right)^2}\\
			\end{array} \right]\in\mathbb{R}^{m_1},\ j\in[T],
		\end{equation*}
		
		\begin{equation*}
			q\left(w\right):=\left[ \begin{array}{c}
				q^{(1)}(w^{(1)})\\ \vdots \\ q^{(T)}(w^{(T)})
			\end{array} \right]\in\mathbb{R}^{Tm_2},\
			q^{(j)}(w^{(j)}):=\left[ \begin{array}{c}
				\frac{ e^{ -\hat{x}_1^ {{\left(j\right)}^\top} w^{\left(j\right)} }
					\left(e^{ -\hat{x}_1^ {{\left(j\right)}^\top} w^{\left(j\right)}}-1\right) } 
				{\left(1+e^{ -\hat{x}_1^ {{\left(j\right)}^\top} w^{\left(j\right)}}\right)^3}\\ \vdots \\
				\frac{ e^{ -\hat{x}_{m_2}^ {{\left(j\right)}^\top} w^{\left(j\right)} }
					\left(e^{ -\hat{x}_{m_2}^ {{\left(j\right)}^\top} w^{\left(j\right)}}-1\right) } 
				{\left(1+e^{ -\hat{x}_{m_2}^ {{\left(j\right)}^\top} w^{\left(j\right)}}\right)^3}\\
			\end{array} \right]\in\mathbb{R}^{m_2}.
		\end{equation*}
		
		For $d=[d^C;d^w]\in\mathbb{R}^{Tn+1}\in \mathcal{C}(v^*,\lambda^*,\mu^*)$, by Proposition \ref{Proposition_cone}, the second-order sufficient condition can be discussed in three cases.
		
		(i) $C^*>0$. For any $d\in \mathcal{C}(v^*,\lambda^*,\mu^*)$, by Proposition \ref{Proposition_cone} (i), it holds that
		\begin{equation*}
			\begin{aligned}
				d^\top\nabla_{vv}^2L\left(v^*;\lambda^*,\mu^*\right)d&=-2(d^C)^2\mu^{*^\top} X{\rm Diag}(u(w^*))X^\top P(v^*)^{-1}g(w^*)\\
				&\qquad+(d^C)^2\left(P(v^*)^{-1}g(w^*)\right)^\top\nabla_{ww}^2 L(v;\lambda^*,\mu^*)P(v^*)^{-1}g(w^*)\\
				&=(d^C)^2z(\mu^*,w^*),\ d^C\in\mathbb{R}.
			\end{aligned}
		\end{equation*}
		By Assumption \ref{Assumption1}, $d^\top\nabla_{vv}^2L\left(v^*;\lambda^*,\mu^*\right)d>0$, for any $d\in \mathcal{C}(v^*,\lambda^*,\mu^*),\ d\neq 0$. Then $v^*$ is a strict local minimizer of problem (\ref{eq_single_mat}).
		
		(ii) $C^*=0$ and $\lambda^*=0$. By (\ref{eq_kkt_ori_d}), we can get $w^*=\mathbf{0}$, and by the definition of $g(w), u(w)$, $p(w)$ and $q(w)$, we get $g(w^*)=-\frac{1}{2}X\hat{y}, u(w^*)=p(w^*)=\mathbf{\frac{1}{4}}$ and $q(w^*)=\mathbf{0}$. Then $L(v;\lambda^*,\mu^*)=\frac{1}{4Tm_1}A^\top A$.
		By Proposition \ref{Proposition_cone} (ii), for any $d\in \mathcal{C}(v^*,\lambda^*,\mu^*)$, it holds that
		\begin{equation*}
			\begin{aligned}
				d^\top\nabla_{vv}^2L\left(v^*;\lambda^*,\mu^*\right)d
				&=\frac{1}{4}(d^C)^2\mu^{*^\top} XX^\top X\hat{y} +\frac{1}{16Tm_1}(d^C)^2\left(X\hat{y}\right)^\top A^\top AX\hat{y}\\
				&= (d^C)^2\iota(\mu^*),\ d^C\ge0.
			\end{aligned}
		\end{equation*}
		By Assumption \ref{Assumption2}, $d^\top\nabla_{vv}^2L\left(v^*;\lambda^*,\mu^*\right)d >0$, for any $d\in \mathcal{C}(v^*,\lambda^*,\mu^*),\ d\neq 0$. Then $v^*$ is a strict local minimizer of problem (\ref{eq_single_mat}).
		
		(iii) $C^*=0$ and $\lambda^*>0$. For any $d\in \mathcal{C}(v^*,\lambda^*,\mu^*),\ d\neq 0$, it automatically holds that  $d^\top\nabla_{vv}^2L\left(v^*;\lambda^*,\mu^*\right)d >0$. Then $v^*$ is a strict local minimizer of problem (\ref{eq_single_mat}).
	\end{proof}
	
	\section{A Squared Smoothing Newton Method}\label{sec4}
	In this section, we will solve the KKT conditions in (\ref{eq_kkt_ori}) by a squared smoothing Newton method. First, we smooth the complementarity conditions with Huber smoothing function \cite{Liang}, and then combine the bi-conjugate gradient method (Bi-CG) \cite{Barrett} and a square smooth Newton method \cite{Liang} to solve the problem. Finally, we prove the superlinear converge rate of the proposed method.
	
	By eliminating $\lambda=-\mu^\top g\left(w\right)$ by (\ref{eq_kkt_ori_a}),  we get the reduced KKT conditions for problem (\ref{eq_single_mat}), that is,\\
	\begin{subnumcases}{\label{eq_kkt_re}}
		0\le\mu^\top g\left(w\right)\bot C\geq0,\label{eq_kkt_re_a}\\
		\frac{1}{Tm_1}\nabla f(w)+\left(I_{Tn}+C\nabla g(w)\right)\mu = \mathbf{0},\label{eq_kkt_re_b}\\
		w+Cg\left(w\right)=\mathbf{0}.\label{eq_kkt_re_c}
	\end{subnumcases}
	
	\subsection{Smoothing system based on Huber smoothing function} \label{subsec4-1}
	(\ref{eq_kkt_re_a}) is a nonsmooth complementarity condition. There are numerous methods to transform this difficult condition to its smoothing equivalent system in the literature. One of the most widely used methods is the smoothing Newton method, which relies on an appropriate smoothing function. In order to deal with the nonsmooth plus function $\max\left\lbrace 0,t\right\rbrace,t\in\mathbb{R}$, the most commonly used smoothing function is Chen-Harker-Kanzow-Smale (CHKS) function, that is, 
	\begin{equation*}
		\xi(\epsilon,t)=\frac{\sqrt{t^2+4\epsilon^2}+t}{2},(\epsilon,t)\in\mathbb{R}\times\mathbb{R},
	\end{equation*}
	which has been extensively studied and used for solving nonlinear complementarity problems \cite{Chen,Kanzow,Smale}. Besides, some other smoothing functions whose properties have also been well-studied \cite{Qi,Sun}. However, it is easy to see that $\xi(\epsilon,t)$ maps any negative number $t$ to a positive one when $\epsilon\neq0$. Thus, it destroys the possible sparsity structure when evaluating the Jacobian of the merit function \cite{Liang}. Hence, smoothing Newton methods based on the CHKS function would require more computational effort. 
	To resolve this issue, we use the Huber smoothing function proposed by Pinar and Zenios \cite{Pinar} which maps any negative number to zero so that the underlying sparsity structure from the plus function is inherited.
	
	It is easy to verify that the following holds
	\begin{equation}\label{eq_plus}
		0\le a \bot b \geq0 \Leftrightarrow b-\max\left\lbrace b-a, 0\right\rbrace=0.
	\end{equation}
	
	By (\ref{eq_plus}), we convert (\ref{eq_kkt_re_a}) to its equivalent equation, that is,
	\begin{equation} \label{eq_max}
		C-\max\left\lbrace C-\mu^\top g\left(w\right),0\right\rbrace=0.
	\end{equation}
	Then we obtain the following equivalent system of (\ref{eq_kkt_re})
	\begin{equation}\label{eq_kkt_max}
		\mathcal{K}(C,w,\mu)=\left[ \begin{array}{c}
			C-\max\left\lbrace C-\mu^\top g\left(w\right),0\right\rbrace\\
			\frac{1}{Tm_1}\nabla f(w) +\left(I_{Tn}+C\nabla g(w)\right)\mu\\
			w+Cg\left(w\right)
		\end{array}\right]=\mathbf{0}.
	\end{equation}
	
	Note that $\rho\left(t\right)=\max\left\lbrace 0,t\right\rbrace, t\in\mathbb{R}$, is not differentiable at $t=0$. We consider its Huber smoothing (or approximation) function \cite{Liang} defined
	as follows:
	\begin{equation}\label{eq_huber}
		h\left(\epsilon,t\right)=
		\begin{cases}
			t-\frac{\left|\epsilon\right|}{2},\ t>\left|\epsilon\right|\\
			\frac{t^2}{2\left|\epsilon\right|},\ 0\le t\le\left|\epsilon\right|\\
			0,\ t<0
		\end{cases},\ 
		\forall\left(\epsilon,t\right)\in\mathbb{R}\backslash\left\lbrace 0\right\rbrace\times\mathbb{R},
		h\left(0,t\right)=\rho(t).
	\end{equation}
	We replace the nonsmooth equation (\ref{eq_max}) by the following smooth equation
	\begin{equation*}
		C-h\left(\epsilon,C-\mu^\top g\left(w\right)\right)=0.
	\end{equation*}
	Following the idea in \cite{Liang}, we solve the following smoothing system for KKT system (\ref{eq_kkt_max})
	\begin{equation}\label{eq_E_hat}
		{\hat{\mathcal{E}}}\left(\epsilon,C,\mu,w\right)=\left[\begin{array}{c}
			\epsilon \\ 
			\left(1+\kappa\left|\epsilon\right|\right)C-h\left(\epsilon,C-\mu^\top g\left(w\right)\right)\\
			\frac{1}{Tm_1}\nabla f(w)+\left(I_{Tn}+C\nabla g(w)\right)\mu\\
			w+Cg\left(w\right)
		\end{array}\right]:=\left[\begin{array}{c}
			\epsilon \\ \mathcal{E}\left(\epsilon,C,\mu,w\right) \end{array}\right]=\mathbf{0},
	\end{equation}
	where $\kappa>0$ is a given constant.
	
	Note that $\mathcal{E}(\cdot)$ is a continuously differentiable system around any $\left(\epsilon,C,\mu,w\right)$ for any $\epsilon>0$. Also, as Liang et al. proved in \cite{Liang}, that it satisfies
	\begin{equation*}
		\mathcal{E}\left(\epsilon^k,C^k,\mu^k,w^k\right)\rightarrow\mathcal{K}\left(C,\mu,w\right),
		\ {\rm as}\ \left(\epsilon^k,C^k,\mu^k,w^k\right)\rightarrow\left(0,C,\mu,w\right).
	\end{equation*}
	
	Note that adding the perturbation terms $\kappa\left|\epsilon\right| C$ for constructing the smoothing function of $\mathcal{K}$ is crucial in our algorithm since it ensures the correctness of our proposed algorithm.
	
	\subsection{A squared smoothing Newton method}\label{subsec4-2}
	In this subsection, we will solve the nonlinear equation system in (\ref{eq_E_hat}) by combining the bi-conjugate gradient method (Bi-CG) \cite{Barrett} and a square smooth Newton method proposed by Liang et al. \cite{Liang}, which can greatly improve the training speed. The details of our algorithm is shown in Algorithm \ref{algo1}, where $J_{\hat{\mathcal{E}}}\left(\epsilon,C,\mu,w\right)$ is the Jacobian matrix of the smoothing system $\hat{\mathcal{E}}(\cdot)$ in (\ref{eq_E_hat}).
	
	\begin{algorithm}
		\caption{A squared smoothing Newton method (SN)}\label{algo1}
		\begin{algorithmic}[1]
			\Require $\hat{\epsilon}\in\left(0,\infty\right),\ r\in\left(0,1\right),\ \hat{\eta}\in\left(0,\infty\right)\ $st.$\delta:=\sqrt{2}\max\left\lbrace r\hat{\epsilon},\hat{\eta}\right\rbrace<1,
			\hat{r}\in\left(0,\infty\right),\ \rho\in\left(0,1\right),\ \sigma\in\left(0,\frac{1}{2}\right),\ \tau\in\left(0,1\right],\  \left(\epsilon^0,C^0,w^0,\mu^0\right)\in\mathbb{R}\times\mathbb{R}\times\mathbb{R}^{Tn}\times\mathbb{R}^{Tn}$.
			
			\For {$k\geq0$}
			
			\If{$\hat{\mathcal{E}}\left(\epsilon^k,C^k,\mu^k,w^k\right)=\mathbf{0}$}
			
			\State Output: $\left(\epsilon^k,C^k,\mu^k,w^k\right)$.
			
			\Else
			
			\State Compute $\eta_k:=\eta\left(\epsilon^k,C^k,\mu^k,w^k\right)$ and $\zeta_k:=\zeta\left(\epsilon^k,C^k,\mu^k,w^k\right)$,
			\Statex \qquad\quad 
			where
			\begin{center}
				\qquad\quad
				$\eta\left(\epsilon^k,C^k,\mu^k,w^k\right)=\min\left\lbrace 1, \hat{r}\|\hat{\mathcal{E}}\left(\epsilon,C,\mu,w\right)\|^\tau\right\rbrace,$\\
				\qquad\quad
				$\zeta\left(\epsilon^k,C^k,\mu^k,w^k\right)=r\min\left\lbrace 1,
				\|\hat{\mathcal{E}}\left(\epsilon,C,\mu,w\right)\|^{1+\tau}\right\rbrace .$
			\end{center}
			
			\State Solve the equation system by bi-conjugate gradient method (Bi-CG):
			\Statex
			\begin{equation}\label{eq_bicg}\qquad\quad
				\hat{\mathcal{E}}\left(\epsilon^k,C^k,\mu^k,w^k\right)+J_{\hat{\mathcal{E}}}\left(\epsilon,C,\mu,w\right)
				\left[ \begin{array}{ccc}\Delta\epsilon\\ \Delta C\\ \Delta\mu\\ \Delta w \end{array}\right]
				=\left[ \begin{array}{ccc}\zeta_k\hat{\epsilon}\\0\\ \mathbf{0}\\ \mathbf{0} \end{array}\right],
			\end{equation}
			\Statex \qquad\quad
			approximately such that
			\begin{center}
				\qquad\quad
				$\|\mathcal{R}_k\|\le\eta_k\|\mathcal{E}\left(\epsilon^k,C^k,\mu^k,w^k\right)+ \nabla_\epsilon\hat{\mathcal{E}} \left(\epsilon^k,C^k,\mu^k,w^k\right)\|,$\\
				\qquad\quad
				$\|\mathcal{R}_k\|\le\hat{\eta_k}\|\hat{\mathcal{E}}\left(\epsilon^k,C^k,\mu^k,w^k\right)\|,$
			\end{center}
			\Statex \qquad\quad where 
			\begin{equation}\label{eq_e}\qquad\quad
				\Delta\epsilon^k:=-\epsilon_k+\zeta_k\hat{\epsilon},
			\end{equation}
			\begin{center}\qquad\quad\qquad\quad
				$\mathcal{R}_k:=\mathcal{E}\left(\epsilon^k,C^k,\mu^k,w^k\right)+J_{\hat{\mathcal{E}}}\left(\epsilon,C,\mu,w\right)
				\left[ \begin{array}{ccc}\Delta\epsilon\\ \Delta C\\ \Delta\mu\\ \Delta w \end{array}\right].$
			\end{center}
			
			\State Compute $\ell_k$ as the smallest non-negative integer $\ell$ satisfying
			\begin{equation}\qquad\quad
				\begin{aligned}
					&\psi\left(\epsilon^k+\rho^\ell\Delta\epsilon^k,C^k+\rho^\ell\Delta C^k, \mu^k+\mu^\ell\Delta\mu^k,w^k+\rho^\ell\Delta w^k \right)\\
					&\le\left[1-2\sigma\left(1-\delta\right)\rho^\ell\right]\psi\left(\epsilon^k,C^k,\mu^k,w^k\right),
				\end{aligned}\label{eq_alg_end}
			\end{equation}
			\Statex \qquad\quad 
			where $\psi\left(\epsilon,C,\mu,w\right)=\|\hat{\mathcal{E}}\left(\epsilon,C,\mu,w\right)\|^2.$

			\State 
			Compute 
			\begin{equation*}\qquad\quad
					\left(\epsilon^{k+1},C^{k+1},\mu^{k+1},w^{k+1}\right)
					=\left(\epsilon^k,C^k,\mu^k,w^k\right)+\rho^{\ell_k}\left( \Delta\epsilon^k,\Delta C^k,\Delta\mu^k,\Delta w^k\right).
			\end{equation*}
			\EndIf
			\EndFor
		\end{algorithmic}
	\end{algorithm}
	
	In the numerical experiments, we can obtain the reduced form (\ref{eq_bicg_re}) by eliminating variable $\Delta\epsilon$ in equation (\ref{eq_bicg}) by (\ref{eq_e}), that is, we apply Bi-CG to solve the following system
	\begin{equation} \label{eq_bicg_re}
		\begin{aligned}
			&\left[ \begin{array}{ccc}
				1+\kappa\epsilon-h_2 & h_2{g(w)}^\top & h_2\mu^\top\nabla g(w)\\
				\nabla g(w)\mu & I_{Tn}+C\nabla g(w) & \nabla_w\left(\frac{1}{Tm_1}\nabla_f(w)+C\nabla g(w)\mu\right)\\
				g\left(w\right) & g\left(w\right) & I_{Tn}+C\nabla g(w)
			\end{array}	\right] 
			\left[ \begin{array}{c}
				\Delta C^k \\ \Delta \mu^k \\ \Delta w^k
			\end{array}\right] 	\\
			=&-
			\left[ \begin{array}{c}
				\left(1+\kappa\left|\epsilon\right|\right)C-h\left(\epsilon,C-\mu^\top g\left(w\right)\right)+(\kappa C-h_1)(\zeta_k\hat{\epsilon}-\epsilon^k)\\
				\frac{1}{Tm_1}\nabla f(w)+\left(I_{Tn}+C\nabla g(w)\right)\mu\\
				w+Cg\left(w\right)
			\end{array}\right] .
		\end{aligned}
	\end{equation}
	\vskip 2cm
	\subsection{Global convergence}\label{subsec4-3}
	In this part, we will investigate the convergence of Algorithm \ref{algo1} (SN). First, we prove that $J_{\hat{\mathcal{E}}}\left(\epsilon,C,\mu,w\right)$ is nonsingular under proper assumptions. Define $\nu(\mu,v,\epsilon)\in\mathbb{R}$ by
	\begin{equation*}
		\nu(\mu,v,\epsilon):=1+\kappa\epsilon-h_2 + h_2g \left(w\right) ^\top P(v)^{-1} \left( \alpha\left(\mu,v\right) P\left(v\right) ^{-1} g\left(w\right) - 2\nabla g\left(w\right) \mu\right).
	\end{equation*}
	\begin{lemma}\label{Lemma1} 
		$J_{\hat{\mathcal{E}}}\left(\epsilon,C,\mu,w\right)$ is nonsingular if $\epsilon>0$ and $z(\mu,v)>-\kappa\epsilon$, where $\kappa>0$ is a given constant.
	\end{lemma}
	
	\begin{proof}
		From (\ref{eq_E_hat}), note that
		\begin{equation}\label{eq_E_hat'}
			\begin{aligned}
				J_{\hat{\mathcal{E}}}\left(\epsilon,C,\mu,w\right)
				& =\left[\begin{array}{cccc}
					1 & 0 & \mathbf{0}_{1\times Tn} & \mathbf{0}_{1\times Tn} \\
					\kappa C-h_1 & 1+\kappa\epsilon-h_2 & h_2{g(w)}^\top & h_2\mu^\top\nabla g(w) \\
					\mathbf{0}_{Tn} & \nabla g(w)\mu & I_{Tn}+C\nabla g(w) & \nabla _w\left(\frac{1}{Tm_1}\nabla f(w)+C\nabla g(w) \mu\right) \\
					\mathbf{0}_{Tn} & g\left(w\right) & \mathbf{0}_{Tn\times Tn} & I_{Tn}+C\nabla g(w)\end{array}\right]\\
				& =\left[\begin{array}{cccc}
					1 & 0 & \mathbf{0}_{1\times Tn} & \mathbf{0}_{1\times Tn} \\
					\kappa C-h_1 & 1+\kappa\epsilon-h_2 & h_2{g(w)}^\top & h_2\mu^\top\nabla g(w) \\
					\mathbf{0}_{Tn} & \nabla g(w)\mu & P(v) & \alpha(\mu,v) \\
					\mathbf{0}_{Tn} & g\left(w\right) & \mathbf{0}_{Tn\times Tn} & P(v)\end{array}\right],
			\end{aligned}
		\end{equation}
		where, $h_1=\nabla_{\epsilon} h\left(\epsilon,t\right),
		\ h_2=\nabla_t h\left(\epsilon,t\right)\in\left(0,1\right],\ \epsilon>0,\ t=C-\mu^\top g\left(w\right).$
		
		To show that $J_{\hat{\mathcal{E}}}\left(\epsilon,C,\mu,w\right)$ is nonsingular, it suffices to show that the following system of linear equations
		\begin{equation}\label{eq_eqE_hat_ori}
			J_{\hat{\mathcal{E}}}\left(\epsilon,C,\mu,w\right)
			\left[ \begin{array}{c}\Delta\epsilon\\ \Delta C\\ \Delta\mu\\ \Delta w \end{array}\right]=\mathbf{0}
		\end{equation}
		only admits zero solution. It is obvious that $\Delta\epsilon=0$. (\ref{eq_eqE_hat_ori}) reduces to the following equations
		\begin{equation}\label{eq_eqE_hat_re}
			\left[\begin{array}{ccc}
				1+\kappa\epsilon-h_2 & h_2{g(w)}^\top & h_2\mu^\top\nabla g(w) \\
				\nabla g(w)\mu & P(v) & \alpha(\mu,v) \\
				g\left(w\right) & \mathbf{0}_{Tn\times Tn} & P(v)
			\end{array}\right]
			\left[ \begin{array}{c} \Delta C\\ \Delta\mu\\ \Delta w \end{array}\right]=\mathbf{0}.
		\end{equation}
		
		We can solve the equation system (\ref{eq_eqE_hat_re}) as follows.
		First, note that $P(v)$ is positive definite, then
		\begin{equation}\label{eq_dw}
			\Delta w = -\Delta C P(v)^{-1} g(w).
		\end{equation}
		By eliminating $\Delta w$ with (\ref{eq_dw}), the equation system (\ref{eq_eqE_hat_re}) is reduced as
		\begin{equation}\label{eq_new0}
			\left[\begin{array}{cc}
				1+\kappa\epsilon-h_2 & h_2{g(w)}^\top \\
				\nabla g(w)\mu & P(v)  
			\end{array}\right]
			\left[ \begin{array}{c} \Delta C\\ \Delta\mu \end{array}\right]=
			\Delta C \left[ \begin{array}{c} 
				h_2\mu^\top\nabla g(w) \\ 
				\alpha(\mu,v) \end{array}\right]P(v)^{-1} g(w).
		\end{equation}
		which gives 
		\begin{equation}\label{eq_dmu}
			\Delta \mu = \Delta C P(v)^{-1} \left( \alpha\left(\mu,v\right) P\left(v\right) ^{-1} g\left(w\right) - \nabla g\left(w\right) \mu\right).
		\end{equation}
		Substituting $\Delta \mu$ into the first equation of (\ref{eq_new0}) gives
		\begin{equation*}
			\begin{aligned}
				& \left( 1+\kappa\epsilon-h_2 + h_2g \left(w\right) ^\top P(v)^{-1} \left( \alpha\left(\mu,v\right) P\left(v\right) ^{-1} g\left(w\right) - \nabla g\left(w\right) \mu\right) \right) \Delta C\\
				&= h_2\mu^\top\nabla g(w) P(v)^{-1} g(w) \Delta C,
			\end{aligned}
		\end{equation*}
		that is,
		\begin{equation}\label{eq_dC}
			\begin{aligned}
				&\left( 1+\kappa\epsilon-h_2 + h_2g \left(w\right) ^\top P(v)^{-1} \left( \alpha\left(\mu,v\right) P\left(v\right) ^{-1} g\left(w\right) - 2\nabla g\left(w\right) \mu\right) \right) \Delta C\\
				&= \left( 1+\kappa\epsilon + h_2\left(z(\mu,v)-1 \right) \right) \Delta C=\nu(\mu,v,\epsilon)\Delta C=0.
			\end{aligned}
		\end{equation}
		
		By $z(\mu,v)>-\kappa\epsilon$ and $h_2\in\left(0,1\right]$, we can get $h_2\left(z(\mu,v)-1 \right)>-1-\kappa\epsilon$, and then it holds that $\nu(\mu,v,\epsilon)\neq0$. The equation system (\ref{eq_dC}) only admits zero solution $\Delta C=0$, leading to $\Delta w=\Delta\mu=\mathbf{0}$ by (\ref{eq_dw}) and (\ref{eq_dmu}). Thus, (\ref{eq_eqE_hat_ori}) only admits zero solution. Therefore, $J_{\hat{\mathcal{E}}}\left(\epsilon,C,\mu,w\right)$ is nonsingular. The proof is completed.
	\end{proof}
	
	By Lemma 9, Lemma 10, Theorem 11 in \cite{Liang} and Lemma \ref{Lemma1} above, we can get the global convergence result as follows. 

	\begin{theorem}\label{thm_cov1}
		SN is a well-defined algorithm and generates an infinite sequence $\left\lbrace \left(\epsilon^k,C^k,\mu^k,w^k\right)\right\rbrace$. Any accumulation point  $\left(\bar{\epsilon},\bar{C},\bar{\mu},\bar{w}\right)$ of $\left\lbrace \left(\epsilon^k,C^k,\mu^k,w^k\right)\right\rbrace$ with $z(\bar{\mu},\bar{v})>0$ is a solution of $\hat{\mathcal{E}}\left(\epsilon,C,\mu,w\right)=\mathbf{0}.$ 
	\end{theorem}

	\begin{proof}
		Lemma 9 and Lemma 10 in \cite{Liang} also hold in this paper since the termination condition of the SN is same as the algorithm proposed in \cite{Liang}, which implies SN is well-defined and generates an infinite sequence. 
		
		$z(\mu,v)$ is continuous, which guarantees that for any accumulation point  $\left(\bar{\epsilon},\bar{C},\bar{\mu},\bar{w}\right)$ with $z(\bar{\mu},\bar{v})>0$, there exists an open neighborhood $\mathcal{U}$ such that $z(\mu,v)>-\kappa\epsilon$ holds for any $(\mu,v)\in \mathcal{U}$, where $\kappa$ and $\epsilon$ are both positive. By Lemma \ref{Lemma1} above, $J_{\hat{\mathcal{E}}}\left(\epsilon,C,\mu,w\right)$ is nonsingular. Note that for $k$ sufficiently large, we have that $\left(\epsilon^k,C^k,\mu^k,w^k\right)$ belongs to $\mathcal{U}$ with $\epsilon^k>0$, leading to $J_{\hat{\mathcal{E}}}\left(\epsilon^k,C^k,\mu^k,w^k\right)$ is nonsingular.
		
		Following the proof in Theorem 11 in \cite{Liang}, we obtain that $\bar{\psi}=0$, which implies that $\hat{\mathcal{E}}\left(\bar{\epsilon},\bar{C},\bar{\mu},\bar{w}\right)=\mathbf{0}$. The proof is completed.
	\end{proof}

	Next, we verified that SN admits a superlinear convergence rate under a proper assumption, which guarantees that $\partial\hat{\mathcal{E}}\left(\epsilon^*,C^*,\mu^*,w^*\right)$ is nonsingular.

	\begin{proposition}\label{Proposition_partial}
		Let $\left(\epsilon^*,C^*,\mu^*,w^*\right)$ be such that $\hat{\mathcal{E}}\left(\epsilon,C,\mu,w\right)=\mathbf{0}$. If $z(\mu^*,v^*)>0$, then every element in $\partial\hat{\mathcal{E}}\left(\epsilon^*,C^*,\mu^*,w^*\right)$ is nonsingular.
	\end{proposition}

	\begin{proof} 
		Let $U\in\partial\hat{\mathcal{E}}\left(\epsilon^*,C^*,\mu^*,w^*\right)$. Since $\epsilon^*=0$, then 
		\begin{equation*}
			U\left[ \begin{array}{c}\Delta\epsilon\\ \Delta C\\ \Delta\mu\\ \Delta w \end{array}\right]=\left[\begin{array}{cccc}
				1 & 0 & \mathbf{0}_{1\times Tn} & \mathbf{0}_{1\times Tn} \\
				\kappa C-h_1^* & 1-h_2^* & h_2^*{g(w^*)}^\top & h_2^*\mu^{*^\top}\nabla g(w^*) \\
				\mathbf{0}_{Tn} & \nabla g(w^*)\mu^* & P(v^*) & \alpha(\mu^*,v^*) \\
				\mathbf{0}_{Tn} & g\left(w^*\right) & \mathbf{0}_{Tn\times Tn} & P(v^*)\end{array}\right]\left[ \begin{array}{c}\Delta\epsilon\\ \Delta C\\ \Delta\mu\\ \Delta w \end{array}\right],
		\end{equation*}
		where $t^*=C^*-\mu^{*^\top} g(w^*),\ h_2^*\in\partial_t h(0,t^*)= \left[0,1\right]$,
		\begin{equation*}\label{eq_huber_e}
			h_1^*\in\partial_\epsilon h(0,t^*) =
			\begin{cases}
				\left[-\frac{1}{2},\frac{1}{2}\right],\ t^*>0,\\
				0,\ t^*\le0.
			\end{cases}
		\end{equation*}
		Here, $\partial_t h(0,t^*)$ and $\partial_\epsilon h(0,t^*)$ denote the generalized gradient of $h(\epsilon,t)$ with respect to $t$ and $\epsilon$ at $(0,t^*)$ respectively.
		
		To show that $U$ is nonsingular, it suffices to show that the following system of linear equations
		\begin{equation}\label{eq_U}
			U\left[ \begin{array}{c}\Delta\epsilon\\ \Delta C\\ \Delta\mu\\ \Delta w \end{array}\right]=\mathbf{0}
		\end{equation}
		only admits zero solution. It is obvious that $\Delta\epsilon=0$. (\ref{eq_U}) reduces to the following equations
		\begin{equation}\label{eq_U_re}
			\left[\begin{array}{ccc}
				1-h_2^* & h_2^*{g(w^*)}^\top & h_2^*\mu^{*^\top}\nabla g(w^*) \\
				\nabla g(w^*)\mu^* & P(v^*) & \alpha(\mu^*,v^*) \\
				g\left(w^*\right) & \mathbf{0}_{Tn\times Tn} & P(v^*)
			\end{array}\right]
			\left[ \begin{array}{c} \Delta C\\ \Delta\mu\\ \Delta w \end{array}\right]=\mathbf{0}.
		\end{equation}
		Similar to the proof in Lemma \ref{Lemma1}, by (\ref{eq_dw}) and (\ref{eq_dmu}), the equation system (\ref{eq_U_re}) can be reduced as
		\begin{equation}\label{eq_h}
			\begin{aligned}
				&\left( 1-h_2^* + h_2^*g \left(w^*\right) ^\top P(v^*)^{-1} \left( \alpha\left(\mu^*,v^*\right) P\left(v^*\right) ^{-1} g\left(w^*\right) - 2\nabla g\left(w^*\right) \mu^*\right) \right) \Delta C\\
				&=(1+h_2^*(z(\mu^*,v^*)-1))\Delta C=0.
			\end{aligned}
		\end{equation}
		By $z(\mu^*,v^*)>0$ and $h_2^*\in\left[0,1\right]$, we can get $1+h_2^*(z(\mu^*,v^*)-1)\neq0$. The equation system (\ref{eq_h}) only admits zero solution $\Delta C=0$, leading to $\Delta\mu=\Delta w=\mathbf{0}$. Thus, (\ref{eq_U}) only admits zero solution. Hence, $U\in\partial\hat{\mathcal{E}}\left(\epsilon^*,C^*,\mu^*,w^*\right)$ is nonsingular.
	\end{proof}

	\begin{theorem}\label{thm_cov2}
		Let $\left(\bar{\epsilon},\bar{C},\bar{\mu},\bar{w}\right)$ be an accumulation point of the infinite sequence $\left\lbrace\left(\epsilon^k,C^k,\mu^k,w^k\right)\right\rbrace$ generated by SN. Suppose that $z(\bar{\mu},\bar{v})>0$. The whole sequence converges to $\left(\epsilon^*,C^*,\mu^*,w^*\right)$ superlinearly, i.e.,
		\begin{equation*}
			\begin{aligned}
				&\left\|\left(\epsilon^{k+1}-\epsilon^\ast,C^{k+1}-C^\ast, \mu^{k+1}-\mu^\ast,w^{k+1}-w^\ast\right)\right\| \\
				&=O\left( \left\|\left(\epsilon^k-\epsilon^\ast, C^k-C^\ast,\mu^k-\mu^\ast,w^k-w^\ast\right)\right\|^{1+\tau}\right),
			\end{aligned}
		\end{equation*}
		where $\tau\in\left(0,1\right]$ is a parameter of SN. Therefore, SN converges superlinearly to its accumulation point $\left(\epsilon^\ast,C^\ast,\mu^\ast,w^\ast\right).$
	\end{theorem}
	
	\begin{proof}
		By Theorem \ref{thm_cov1}, we see that $\hat{\mathcal{E}}\left(\bar{\epsilon},\bar{C},\bar{\mu},\bar{w}\right)=\mathbf{0}$ and in particular, $\bar{\epsilon}=0$.
		By Proposition \ref{Proposition_partial}, every element in $\partial{\hat{\mathcal{E}}}\left(0,\bar{C},\bar{\mu},\bar{w}\right)$ is nonsingular with $z(\bar{\mu},\bar{v})>0$. By Proposition 3.1 in \cite{Qi_b}, for all $k$ sufficiently large, it holds that 
		\begin{equation}
			\left\|J_{\hat{\mathcal{E}}}\left(\epsilon^k,C^k,\mu^k,w^k\right)^{-1}\right\|=O(1).
		\end{equation}
	
		Denote $\bar{b}:=\left(\bar{\epsilon},\bar{C},\bar{\mu},\bar{w}\right)^\top,\ b^k:=\left(\epsilon^k,C^k,\mu^k,w^k\right)^\top,\ \hat{\mathcal{E}}_*:=\hat{\mathcal{E}}\left(\bar{\epsilon},\bar{C},\bar{\mu},\bar{w}\right),\ \Delta b^k:=\left(\Delta \epsilon^k,\Delta C^k,\Delta \mu^k,\Delta w^k\right)^\top,\  \hat{\mathcal{E}}_k:=\hat{\mathcal{E}}\left(\epsilon^k, C^k,\mu^k,w^k\right)$ and $J_k=J_{\hat{\mathcal{E}}}\left(\epsilon^k,C^k,\mu^k,w^k\right)$ for $k\ge0$. Then, we can verify that 
		\begin{equation}\label{eq_con0}
			\begin{aligned}
				\left\|b^k+\Delta b^k-\bar{b}\right\|&=\left\|b^k+J_k^{-1}\left(\left(\begin{array}{c}\zeta_k \hat{\epsilon} \\ \mathcal{R}_k \end{array}\right) -\hat{\mathcal{E}}_k\right)\right\|\\
				&=\left\|-J_k^{-1}\left(\hat{\mathcal{E}}_k-J_k(b^k-\bar{b})-\left(\begin{array}{c} \zeta_k \hat{\epsilon} \\ \mathcal{R}_k \end{array}\right) \right)\right\|\\
				&=O\left(\left\|\hat{\mathcal{E}}_k- J_k(b^k-\bar{b})\right\|  \right) +O\left( \left\|\hat{\mathcal{E}}_k\right\|^{1+\tau}\right) +O\left( \left\| \mathcal{R}_k\right\| \right) 
			\end{aligned}
		\end{equation}
		Since $h(\epsilon,t)$ is strongly semismooth everywhere \cite{Liang},  $\hat{\mathcal{E}}$ is also strongly semismooth at $\bar{b}$. Thus, for $k$ sufficiently large, it holds that 
		\begin{equation}\label{eq_con1}
			\left\|\hat{\mathcal{E}}_k- J_k(b^k-\bar{b})\right\|=O\left( \left\| b^k-\bar{b}\right\|^2 \right) .
		\end{equation}
		Since $\hat{\mathcal{E}}$ is locally Lipschitz continuous at $\bar{b}$, for all $k$ sufficiently large, it holds that 
		\begin{equation}\label{eq_con2}
			\left\|\hat{\mathcal{E}}_k\right\|=\left\|\hat{\mathcal{E}}_k-\hat{\mathcal{E}}_*\right\|=O\left(\left\| b^k-\bar{b}\right\|  \right) .
		\end{equation}
		By Theorem 14 in \cite{Liang}, we can see that 
		\begin{equation*}
			\left\|\mathcal{R} \right\|\le O\left( \left\| b^k-\bar{b}\right\|^{1+\tau} \right) ,\ \tau\in(0,1]
		\end{equation*}
		which together with (\ref{eq_con0}),(\ref{eq_con1}) and (\ref{eq_con2}) implies that 
		\begin{equation*}
			\left\|b^k+\Delta b^k-\bar{b}\right\|= O\left( \left\| b^k-\bar{b}\right\|^{1+\tau} \right) ,\ \tau\in(0,1].
		\end{equation*}
	
		Therefore, we get for $k$ sufficiently large that
		\begin{equation*}
			\begin{aligned}
				\psi(b^k+\Delta b^k)&=\left\| \hat{\mathcal{E}}(b^k+\Delta b^k) \right\|^2
				=O\left(\left\| b^k+\Delta b^k-\bar{b}\right\| ^2 \right) 
				=O\left(\left\|b^k-\bar{b}\right\|^{2\left(1+\tau \right)}\right)\\
				&=O\left(\left\|\hat{\mathcal{E}}_k\right\|^{2\left(1+\tau \right) }  \right)
				=O\left(\left\|\psi(b^k)\right\|^{1+\tau}  \right)
				=o\left(\left\|\psi(b^k)\right\| \right),\ \tau\in(0,1]
			\end{aligned}
		\end{equation*}
		which shows that, for $k$ sufficient large, $b^{k+1}=b^k+\Delta b^{k}$, i.e., the unit step is eventually accepted with a given constant $\tau\in(0,1]$. The proof is completed.		
	\end{proof}
	
	\section{Numerical Results}\label{sec5}
	In this section, we present the cross-validation algorithm for selecting the hyperparameter $C$ in SVC, as shown in Algorithm 2. All the numerical tests are conducted in Matlab R2018b. All the data sets are collected from the LIBSVM library: https://www.csie.ntu.edu.tw/cjlin/
	libsvmtools/datasets/. The data descriptions are shown in Table 1.
	\vskip 4mm
	\begin{algorithm}
		\caption{The cross-validation algorithm (SN-CV)}\label{algo2}
		\begin{algorithmic}[1]
			\vskip 1mm
			\State 
			Given $T$, split the data set into a subset $\Omega$ with $l_1$ points and a hold-out test set $\Theta$ with $l_2$ points. The set $\Omega$ is equally partitioned into $T$ pairwise disjoint subsets, one for each fold.
			\vskip 1mm
			\State 
			$\textbf{Select}$ an optimal hyperparameter $C$ by SN in Algorithm 1.
			\vskip 1mm
			\State
			$\textbf{Post-processing procedure.}$ The regularization hyperparameter $\hat{C}$ is rescaled by a factor $T/(T-1)$. This gives the final classifier $\hat{w}$.
			\vskip 1mm
		\end{algorithmic}
	\end{algorithm}
	
	\begin{table}[htb]\label{tab_Descriptions}
		\renewcommand{\arraystretch}{1.15}
		\setlength{\abovecaptionskip}{0cm}
		\caption{Descriptions of data sets}
		\begin{center}
			\begin{tabular}{ c p{6.5em} r r r p{0.8em} c p{4em} r r r } 
				\specialrule{0.1em}{0pt}{0pt}
				No. & Data set & $l_1\ $ & $l_2\ $ & $n\ $ & &
				No. & Data set & $l_1\ $ & $l_2\ $ & $n\ $ \\ 
				\specialrule{0.1em}{0pt}{0pt}
				1 & fourclass & 735 & 127 & 2 & & 11 & a2a & 1800 & 465 & 119 \\
				2 & diabetes & 540 & 228 & 8 & & 12 &  a3a & 1200 & 985 & 122 \\ 
				3 & breast-cancer & 540 & 143 & 10 & & 13 & a4a & 1200 & 581 & 122 \\
				4 & heart & 162 & 108 & 13 & & 14 & a6a & 2250 & 8970 & 122 \\
				5 & australian & 600 & 90 & 14 & & 15 & a7a & 2700 & 13400 & 122 \\
				6 & svmguide4 & 30 & 11 & 22 & & 16 & a9a & 3300 & 29261 & 123 \\
				7 & german.number & 600 & 400 & 24 & & 17 & w1a & 1500 & 977 & 300 \\
				8 &ionosphere & 243 & 108 & 34 & & 18& w2a & 1800 & 1670 & 300 \\
				9 &sonar & 150 & 58 & 60 & & 19 & w4a & 3000 & 4366 & 300 \\
				10 & phishing & 1500 & 500 & 68 & & 20 & w5a & 3000 & 6888 & 300 \\
				\specialrule{0.1em}{0pt}{0pt}
			\end{tabular}
		\end{center}
	\end{table}
	
	Our method is denoted as SN. We use imSN (exSN) to denote the Jacobian matrix in (\ref{eq_bicg}) is formed implicitly (explicitly). We compare our methods  with three other approaches: SNO-CV, where the optimal hyperparameter $C$ is selected by the SNOPT solver \cite{Gill}; GS-UNC and GS-SNO, in both of which the hyperparameter $C$ is searched in a grid range \cite{Couellan,Kunapuli_a}
	\begin{center}
		$C\in \left\lbrace \right.  \left. 0.5\times{10}^{-4},{10}^{-4}, 0.5\times{10}^{-3},{10}^{-3}, 0.5\times{10}^{-2},{10}^{-2}, 0.5\times{10}^{-1},{10}^{-1},\right.$\\
		$\quad\ \ \left. 0.5, 1, 0.5\times{10}^1, {10}^1, 0.5\times{10}^2, {10}^2, 0.5\times{10}^3, {10}^3, 0.5\times{10}^4, {10}^4 \  \right\rbrace ,$
	\end{center}
	and the lower-level problems defined as (\ref{eq_lower_ori}) are solved using the given $C$. In GS-UNC, the optimization problem is solved by the MATLAB function $fminunc$, while in GS-SNO, the same problem is solved by SNOPT solver.
	
	The parameters of three methods are set as follows. 
	For imSN and exSN, we set the initial values as $\ \epsilon^0=1,\ C^0=1,\ \mu^0=0,\ w^0=0,\ r=\hat{r}=0.6,\ \hat{\eta}=0.2,\ \rho=0.5,\ \sigma=1\times{10}^{-8},\ \tau=0.2,\ \hat{\epsilon}=0.5, \ \kappa=1$.
	For SNO-CV, we set $\ C^0=1,\ \mu^0=\mathbf{0},\ w^0=\mathbf{0}$. 
	\vskip 5mm
	We compare the aforementioned methods in the following three aspects:\\
	1. Training time ($t$), which is the time of classification training with the CV set.\\
	2. Test error ($E_t\left(\%\right) $) as defined by 
	\begin{center}
		$E_t=\frac{1}{l_2}\sum \limits_{x,y\in\Theta} \frac{1}{2} \left| \ {\rm sign}\left( w^{*^\top} x \right) -y \ \right| ,$
	\end{center}
	where $\Theta$ is the test set.\\
	3. Cross validation error ($E_{CV}$) as defined in the objective function of problem  (\ref{eq_NLP}).
	
	The results are reported in Table \ref{tab_Results}, where we mark the winners in bold and they are also illustrated in Figures 1-3. We have the following observations. 
	
	Firstly, for training time summarized also in Figure \ref{Fig. t}, our SN methods take the least time among the five methods in almost all data sets. Especially for the datasets with a large number of features (No.10 to No.20), the training time of the two SN methods are much less than that of other remaining methods. Besides, we can find that there is a structural difference caused by the number of sample feature. The training time of exSN is slightly less than that of imSN when features are less than about 60, but imSN is more efficient in large datasets with features more than 60. This reflects the advantages of implicit methods in large-scale program design.
	Secondly, the methods based on bilevel optimization (i.e., imSN, exSN, SNO-CV) performs better than the two grid searching methods (i.e., GS-UNC, GS-SNO) in terms of test error ($E_t(\%)$), implying that our bilevel optimization approaches are more capable of classification prediction. 
	Thirdly, in terms of cross validation error ($E_{CV}$), SNO-CV is the best, and the two SN methods are very close to SNO-CV in most data sets. 
	Finally, Figure \ref{Fig. Vio} demonstrates the typical decrease of $\|\mathcal{\hat{E}}\|$ for some datasets, where one can indeed observe the superlinear convergence rate of our algorithm.
	
	Overall, the two SN methods are able to achieve satisfactory results, which are competitive with SNO-CV, and far superior to two GS methods. Meanwhile, the training time can be reduced a lot comparing with SNO-CV which is a widely used bilevel optimization method. 
	
	\begin{figure}[htbp]
		\centering
		\begin{minipage}{1\linewidth}
			\centering
			\includegraphics[width=1\textwidth]{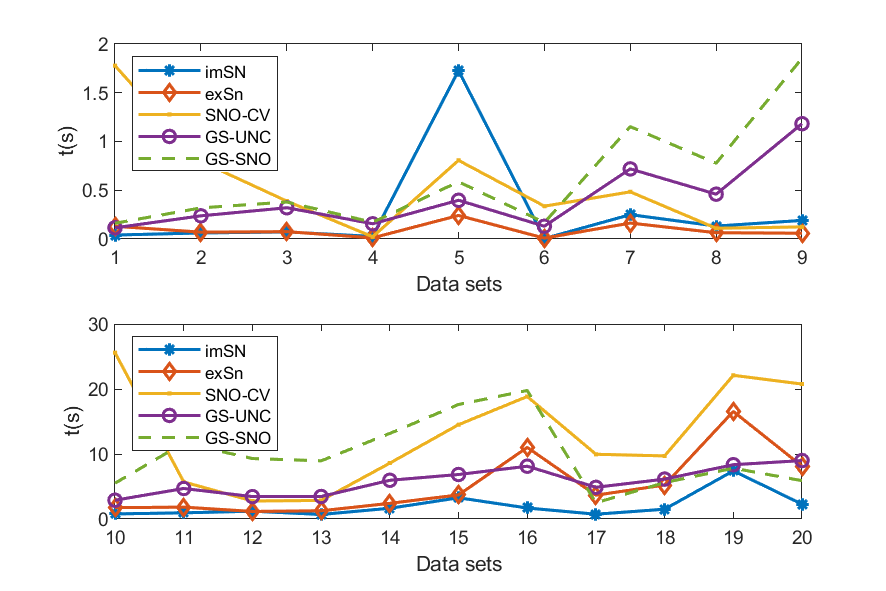}
			\setlength{\abovecaptionskip}{-0.5cm}
			\caption{The comparison among the three methods on training time ($t$)} 
			\label{Fig. t} 
		\end{minipage}
		\vskip 0.5cm
		\begin{minipage}{1\linewidth}
			\centering
			\includegraphics[width=1\textwidth]{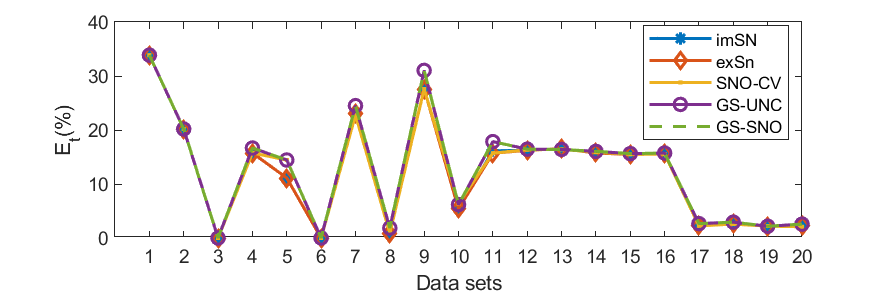}
			\setlength{\abovecaptionskip}{-0.3cm}
			\caption{The comparison among the three methods on test error ($E_t\left(\%\right) $)} 
			\label{Fig. e} 
		\end{minipage}
		\vskip 0.5cm
		\begin{minipage}{1\linewidth}
			\centering
			\includegraphics[width=1\textwidth]{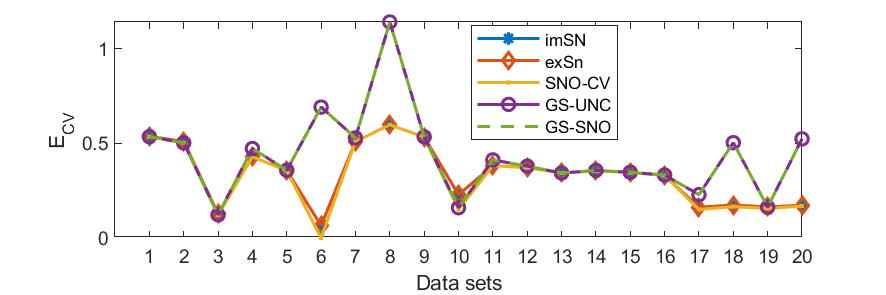}
			\setlength{\abovecaptionskip}{-0.3cm}
			\caption{The comparison among the three methods on cross validation error ($E_{CV}$)} 
			\label{Fig. c} 
		\end{minipage}
		\vskip 0.5cm
	\end{figure}
	
	\begin{center}
		\begin{longtable}{ c l c l r r r r}			
			\caption{Computational results for $T$ = 3} \label{tab_Results} \\
			\specialrule{0.1em}{0pt}{0pt}
			\multicolumn{1}{c}{No.} & 
			\multicolumn{1}{l}{Data set} & 
			\multicolumn{1}{c}{(m,n,$\gamma$)} &
			\multicolumn{1}{l}{Method} &
			\multicolumn{1}{c}{$C$} &
			\multicolumn{1}{c}{$t$} & 
			\multicolumn{1}{c}{$E_t\left(\%\right)$} & 
			\multicolumn{1}{c}{$E_{CV}$} \\ \specialrule{0.1em}{0pt}{0pt}
			\endfirsthead
			
			\specialrule{0em}{15pt}{2pt}
			\multicolumn{8}{c}
			{{\bfseries \tablename\ \thetable{} -- continued from previous page}} \\
			\specialrule{0.1em}{0pt}{0pt}
			\multicolumn{1}{c}{No.} & 
			\multicolumn{1}{l}{Data set} &
			\multicolumn{1}{c}{(m,n,$\gamma$)} &
			\multicolumn{1}{l}{Method} &
			\multicolumn{1}{c}{$C$} &
			\multicolumn{1}{c}{$t$} & 
			\multicolumn{1}{c}{$E_t\left(\%\right)$} & 
			\multicolumn{1}{c}{$E_{CV}$} \\ \specialrule{0.1em}{0pt}{0pt}
			\endhead
			
			\hline \multicolumn{8}{r}{Continued on next page} \\ 
			\endfoot
			
			\hline \hline
			\endlastfoot
			
			1	&	fourclass	&	$(	7	,	6	,	14	)$	&	imSN	&	0.468	&	$\mathbf{	0.041 	}$	&	$\mathbf{	33.858 	}$	&		0.5373 		\\	
			&		&								&	exSN	&	0.468	&		0.130 		&	$\mathbf{	33.858 	}$	&		0.5373 		\\	
			&		&								&	SNO-CV	&	3.004	&		1.777 		&	$\mathbf{	33.858 	}$	&	$\mathbf{	0.5354 	}$	\\	
			&		&								&	GS-UNC	&	1.5	&		0.115 		&	$\mathbf{	33.858 	}$	&		0.5355 		\\	
			&		&								&	GS-SNO	&	1.5	&		0.161 		&	$\mathbf{	33.858 	}$	&		0.5355 		\\	\hline
			2	&	diabetes	&	$(	25	,	24	,	50	)$	&	imSN	&	0.800	&	$\mathbf{	0.062 	}$	&	$\mathbf{	20.175 	}$	&		0.5094 		\\	
			&		&								&	exSN	&	0.800	&		0.070 		&	$\mathbf{	20.175 	}$	&		0.5094 		\\	
			&		&								&	SNO-CV	&	4.688	&		0.795 		&	$\mathbf{	20.175 	}$	&	$\mathbf{	0.5025 	}$	\\	
			&		&								&	GS-UNC	&	7.5	&		0.238 		&	$\mathbf{	20.175 	}$	&		0.5026 		\\	
			&		&								&	GS-SNO	&	7.5	&		0.319 		&	$\mathbf{	20.175 	}$	&		0.5026 		\\	\hline
			3	&	breast-	&	$(	31	,	30	,	62	)$	&	imSN	&	0.693	&	$\mathbf{	0.072 	}$	&	$\mathbf{	0.000 	}$	&		0.1284 		\\	
			&	cancer	&								&	exSN	&	0.693	&		0.075 		&	$\mathbf{	0.000 	}$	&		0.1284 		\\	
			&		&								&	SNO-CV	&	3.528	&		0.389 		&	$\mathbf{	0.000 	}$	&	$\mathbf{	0.1217 	}$	\\	
			&		&								&	GS-UNC	&	7.5	&		0.321 		&	$\mathbf{	0.000 	}$	&		0.1225 		\\	
			&		&								&	GS-SNO	&	7.5	&		0.379 		&	$\mathbf{	0.000 	}$	&		0.1225 		\\	\hline
			4	&	heart	&	$(	40	,	39	,	80	)$	&	imSN	&	0.483	&		0.026 		&	$\mathbf{	15.741 	}$	&	$\mathbf{	0.4291 	}$	\\	
			&		&								&	exSN	&	0.483	&	$\mathbf{	0.012 	}$	&	$\mathbf{	15.741 	}$	&	$\mathbf{	0.4291 	}$	\\	
			&		&								&	SNO-CV	&	0.499	&		0.023 		&	$\mathbf{	15.741 	}$	&	$\mathbf{	0.4291 	}$	\\	
			&		&								&	GS-UNC	&	0.075	&		0.156 		&		16.667 		&		0.4727 		\\	
			&		&								&	GS-SNO	&	0.075	&		0.172 		&		16.667 		&		0.4727 		\\	\hline
			5	&	australian	&	$(	43	,	42	,	86	)$	&	imSN	&	0.328	&		1.727 		&	$\mathbf{	11.111 	}$	&		0.3568 		\\	
			&		&								&	exSN	&	0.400	&	$\mathbf{	0.242 	}$	&	$\mathbf{	11.111 	}$	&		0.3563 		\\	
			&		&								&	SNO-CV	&	1.287	&		0.807 		&		14.444 		&	$\mathbf{	0.3552 	}$	\\	
			&		&								&	GS-UNC	&	7.5	&		0.397 		&		14.444 		&		0.3578 		\\	
			&		&								&	GS-SNO	&	7.5	&		0.587 		&		14.444 		&		0.3578 		\\	\hline
			6	&	svmguide4	&	$(	67	,	66	,	134	)$	&	imSN	&	1.288	&	$\mathbf{	0.005 	}$	&	$\mathbf{	0.000 	}$	&		0.0653 		\\	
			&		&								&	exSN	&	1.288	&		0.009 		&	$\mathbf{	0.000 	}$	&		0.0653 		\\	
			&		&								&	SNO-CV	&	49509	&		0.337 		&	$\mathbf{	0.000 	}$	&	$\mathbf{	0.0000 	}$	\\	
			&		&								&	GS-UNC	&	0.00015	&		0.131 		&	$\mathbf{	0.000 	}$	&		0.6918 		\\	
			&		&								&	GS-SNO	&	0.00015	&		0.169 		&	$\mathbf{	0.000 	}$	&		0.6918 		\\	\hline
			7	&	german.	&	$(	73	,	72	,	146	)$	&	imSN	&	0.226	&		0.250 		&		23.000 		&		0.5094 		\\	
			&	number	&								&	exSN	&	0.246	&	$\mathbf{	0.164 	}$	&		23.000 		&		0.5095 		\\	
			&		&								&	SNO-CV	&	0.196	&		0.484 		&	$\mathbf{	22.750 	}$	&	$\mathbf{	0.5092 	}$	\\	
			&		&								&	GS-UNC	&	1500	&		0.719 		&		24.500 		&		0.5292 		\\	
			&		&								&	GS-SNO	&	1500	&		1.152 		&		24.500 		&		0.5292 		\\	\hline
			8	&	ionosphere	&	$(	103	,	102	,	206	)$	&	imSN	&	0.135	&		0.133 		&	$\mathbf{	0.926 	}$	&	$\mathbf{	0.5980 	}$	\\	
			&		&								&	exSN	&	0.135	&	$\mathbf{	0.064 	}$	&	$\mathbf{	0.926 	}$	&	$\mathbf{	0.5980 	}$	\\	
			&		&								&	SNO-CV	&	0.123	&		0.110 		&	$\mathbf{	0.926 	}$	&	$\mathbf{	0.5980 	}$	\\	
			&		&								&	GS-UNC	&	15	&		0.459 		&		1.852 		&		1.1415 		\\	
			&		&								&	GS-SNO	&	15	&		0.777 		&		1.852 		&		1.1415 		\\	\hline
			9	&	sonar	&	$(	181	,	180	,	362	)$	&	imSN	&	0.566	&		0.191 		&	$\mathbf{	27.586 	}$	&	$\mathbf{	0.5333 	}$	\\	
			&		&								&	exSN	&	0.566	&	$\mathbf{	0.059 	}$	&	$\mathbf{	27.586 	}$	&	$\mathbf{	0.5333 	}$	\\	
			&		&								&	SNO-CV	&	0.572	&		0.124 		&	$\mathbf{	27.586 	}$	&	$\mathbf{	0.5333 	}$	\\	
			&		&								&	GS-UNC	&	0.75	&		1.183 		&		31.034 		&		0.5350 		\\	
			&		&								&	GS-SNO	&	0.75	&		1.853 		&		31.034 		&		0.5350 		\\	\hline
			10	&	phishing	&	$(	205	,	204	,	410	)$	&	imSN	&	1.632	&	$\mathbf{	0.781 	}$	&	$\mathbf{	5.400 	}$	&		0.2289 		\\	
			&		&								&	exSN	&	1.632	&		1.748 		&	$\mathbf{	5.400 	}$	&		0.2289 		\\	
			&		&								&	SNO-CV	&	96.131	&		25.584 		&		6.400 		&	$\mathbf{	0.1505 	}$	\\	
			&		&								&	GS-UNC	&	15	&		2.890 		&		6.200 		&		0.1600 		\\	
			&		&								&	GS-SNO	&	15	&		5.511 		&		6.200 		&		0.1600 		\\	\hline
			11	&	a2a	&	$(	358	,	357	,	716	)$	&	imSN	&	0.254	&	$\mathbf{	0.949 	}$	&		16.129 		&		0.3824 		\\	
			&		&								&	exSN	&	0.284	&		1.825 		&	$\mathbf{	15.699 	}$	&	$\mathbf{	0.3823 	}$	\\	
			&		&								&	SNO-CV	&	0.303	&		5.704 		&	$\mathbf{	15.699 	}$	&	$\mathbf{	0.3823 	}$	\\	
			&		&								&	GS-UNC	&	75	&		4.693 		&		17.849 		&		0.4126 		\\	
			&		&								&	GS-SNO	&	75	&		11.684 		&		17.849 		&		0.4126 		\\	\hline
			12	&	a3a	&	$(	367	,	366	,	734	)$	&	imSN	&	0.414	&		1.200 		&	$\mathbf{	16.244 	}$	&		0.3722 		\\	
			&		&								&	exSN	&	0.347	&	$\mathbf{	1.161 	}$	&	$\mathbf{	16.244 	}$	&		0.3727 		\\	
			&		&								&	SNO-CV	&	0.611	&		2.774 		&		16.345 		&	$\mathbf{	0.3717 	}$	\\	
			&		&								&	GS-UNC	&	0.15	&		3.445 		&		16.447 		&		0.3795 		\\	
			&		&								&	GS-SNO	&	0.15	&		9.300 		&		16.447 		&		0.3795 		\\	\hline
			13	&	a4a	&	$(	367	,	366	,	734	)$	&	imSN	&	0.458	&	$\mathbf{	0.700 	}$	&		16.532 		&		0.3443 		\\	
			&		&								&	exSN	&	0.460	&		1.265 		&		16.532 		&		0.3443 		\\	
			&		&								&	SNO-CV	&	0.747	&		2.849 		&	$\mathbf{	16.392 	}$	&	$\mathbf{	0.3433 	}$	\\	
			&		&								&	GS-UNC	&	0.75	&		3.464 		&	$\mathbf{	16.392 	}$	&	$\mathbf{	0.3433 	}$	\\	
			&		&								&	GS-SNO	&	0.75	&		8.952 		&	$\mathbf{	16.392 	}$	&	$\mathbf{	0.3433 	}$	\\	\hline
			14	&	a6a	&	$(	367	,	366	,	734	)$	&	imSN	&	0.300	&	$\mathbf{	1.651 	}$	&	$\mathbf{	15.753 	}$	&		0.3560 		\\	
			&		&								&	exSN	&	0.364	&		2.412 		&	$\mathbf{	15.753 	}$	&		0.3554 		\\	
			&		&								&	SNO-CV	&	0.544	&		8.600 		&		15.909 		&	$\mathbf{	0.3550 	}$	\\	
			&		&								&	GS-UNC	&	1.5	&		5.956 		&		16.042 		&		0.3571 		\\	
			&		&								&	GS-SNO	&	1.5	&		13.167 		&		16.054 		&		0.3571 		\\	\hline
			15	&	a7a	&	$(	367	,	366	,	734	)$	&	imSN	&	0.388	&	$\mathbf{	3.263 	}$	&		15.507 		&		0.3466 		\\	
			&		&								&	exSN	&	0.342	&		3.705 		&	$\mathbf{	15.463 	}$	&		0.3469 		\\	
			&		&								&	SNO-CV	&	0.649	&		14.512 		&		15.493 		&		0.3460 		\\	
			&		&								&	GS-UNC	&	1.5	&		6.846 		&		15.627 		&		0.3472 		\\	
			&		&								&	GS-SNO	&	1.5	&		17.617 		&		15.627 		&		0.3472 		\\	\hline
			16	&	a9a	&	$(	370	,	369	,	740	)$	&	imSN	&	0.269	&	$\mathbf{	1.693 	}$	&		15.587 		&		0.3302 		\\	
			&		&								&	exSN	&	0.278	&		11.014 		&		15.584 		&		0.3301 		\\	
			&		&								&	SNO-CV	&	0.359	&		18.841 		&	$\mathbf{	15.546 	}$	&	$\mathbf{	0.3299 	}$	\\	
			&		&								&	GS-UNC	&	0.15	&		8.121 		&		15.731 		&		0.3327 		\\	
			&		&								&	GS-SNO	&	0.15	&		19.773 		&		15.731 		&		0.3327 		\\	\hline
			17	&	w1a	&	$(	901	,	900	,	1802	)$	&	imSN	&	0.857	&	$\mathbf{	0.737 	}$	&		2.354 		&		0.1625 		\\	
			&		&								&	exSN	&	0.857	&		3.667 		&		2.354 		&		0.1625 		\\	
			&		&								&	SNO-CV	&	4.153	&		9.959 		&	$\mathbf{	2.252 	}$	&	$\mathbf{	0.1506 	}$	\\	
			&		&								&	GS-UNC	&	0.075	&		4.877 		&		2.661 		&		0.2290 		\\	
			&		&								&	GS-SNO	&	0.075	&		2.494 		&		2.661 		&		0.2290 		\\	\hline
			18	&	w2a	&	$(	901	,	900	,	1802	)$	&	imSN	&	0.679	&	$\mathbf{	1.506 	}$	&		2.635 		&		0.1741 		\\	
			&		&								&	exSN	&	0.679	&		5.229 		&		2.635 		&		0.1741 		\\	
			&		&								&	SNO-CV	&	2.602	&		9.706 		&	$\mathbf{	2.575 	}$	&	$\mathbf{	0.1655 	}$	\\	
			&		&								&	GS-UNC	&	0.0015	&		6.137 		&		2.934 		&		0.5036 		\\	
			&		&								&	GS-SNO	&	0.0015	&		5.602 		&		2.934 		&		0.5036 		\\	\hline
			19	&	w4a	&	$(	901	,	900	,	1802	)$	&	imSN	&	0.695	&	$\mathbf{	7.408 	}$	&	$\mathbf{	2.176 	}$	&		0.1635 		\\	
			&		&								&	exSN	&	0.667	&		16.549 		&	$\mathbf{	2.176 	}$	&		0.1640 		\\	
			&		&								&	SNO-CV	&	2.303	&		22.107 		&		2.199 		&	$\mathbf{	0.1564 	}$	\\	
			&		&								&	GS-UNC	&	0.75	&		8.352 		&	$\mathbf{	2.176 	}$	&		0.1627 		\\	
			&		&								&	GS-SNO	&	0.75	&		7.810 		&		2.176 		&		0.1627 		\\	\hline
			20	&	w5a	&	$(	901	,	900	,	1802	)$	&	imSN	&	0.626	&	$\mathbf{	2.281 	}$	&		2.236 		&		0.1729 		\\	
			&		&								&	exSN	&	0.626	&		8.128 		&		2.236 		&		0.1729 		\\	
			&		&								&	SNO-CV	&	1.685	&		20.758 		&	$\mathbf{	2.076 	}$	&	$\mathbf{	0.1685 	}$	\\	
			&		&								&	GS-UNC	&	0.00075	&		8.990 		&		2.570 		&		0.5239 		\\	
			&		&								&	GS-SNO	&	0.00075	&		5.899 		&		2.570 		&		0.5239 		\\	\hline
			
		\end{longtable}
	\end{center}
	
	\begin{figure}[htbp]
		\centering
		\setlength{\abovecaptionskip}{0.1cm}
		\begin{minipage}{0.49\linewidth}
			\centering
			\includegraphics[width=1.05\linewidth]{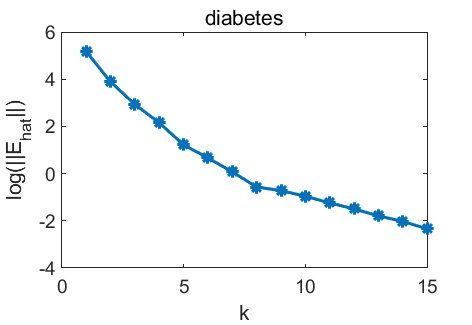}
		\end{minipage}
		\begin{minipage}{0.49\linewidth}
			\centering
			\includegraphics[width=1.05\linewidth]{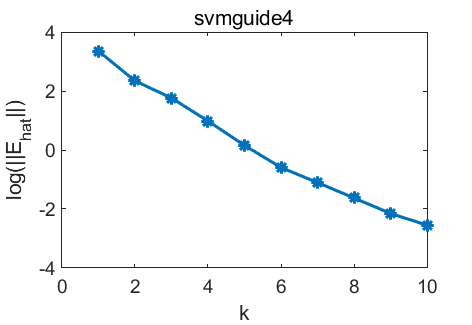}
		\end{minipage}
		
		\begin{minipage}{0.49\linewidth}
			\centering
			\includegraphics[width=1.05\linewidth]{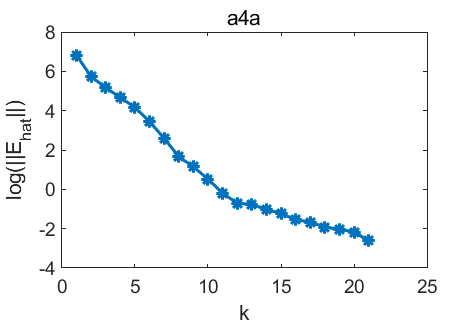}
		\end{minipage}
		\begin{minipage}{0.49\linewidth}
			\centering
			\includegraphics[width=1.05\linewidth]{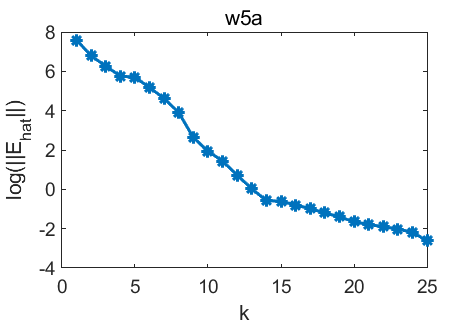}
		\end{minipage}
		\caption{$log\|\hat{\mathcal{E}}\|$ along iterations by imSN} 
		\label{Fig. Vio} 
	\end{figure}
	\vskip 2cm
	Next, we report some further information of the solutions computed by two SN methods in Table \ref{tab_FurtherInfo}.
	(i) The number of iterations in outer loops (SN) and inner loops (BiCG).
	(ii) $z(\mu^*,v^*)$ given in (\ref{eq_z}).
	(iii) The violation of the nonlinear system in terms of $\|\hat{\mathcal{E}}\|_2,\|\epsilon\|_2,$ and $\|\mathcal{E}\|_2$.
	
	One can observe from Table \ref{tab_FurtherInfo} that for most data set, the outer loop reaches the termination condition within 50 iterations. In Fig. \ref{Fig. Vio}, one can see that $\log\|\hat{\mathcal{E}}\|$ decreases almost linearly along $k$ and stopped successfully within small iterations, indicating the superlinear convergence rate of SN. Note that all the returned solution $C^*$s by SN methods are always positive. In terms of $z(\mu^*,w^*)$, it can be seen that $z(\mu^*,v^*)$ is always positive, which satisfies Assumption \ref{Assumption1} and the condition in Theorem \ref{thm_cov1} and Theorem \ref{thm_cov2}, implying that SN provides strict local minimizers of (\ref{eq_NLP}) and the superlinear convergence rate is achieved. Finally, from the last three columns in Table \ref{tab_FurtherInfo}, one can see that the nonlinear system is successfully solved since the returned violation of he system (\ref{eq_E_hat}) is always smaller than 0.1.

	\vskip 0.5cm
	\begin{center}
		\begin{longtable}{ c l l c c c c c }
			\caption{Further information of SN} \label{tab_FurtherInfo} \\
			\specialrule{0.1em}{0pt}{0pt}
			\multicolumn{1}{c}{No.} & 
			\multicolumn{1}{l}{Data set} & 
			\multicolumn{1}{l}{Method} &
			\multicolumn{1}{c}{($k,iter$)} &
			\multicolumn{1}{c}{$z(\mu^*,v^*)$} &
			\multicolumn{1}{c}{$\|\hat{\mathcal{E}}\|_2$} & 
			\multicolumn{1}{c}{$\|\epsilon\|_2$} &
			\multicolumn{1}{c}{$\|\mathcal{E}\|_2$}\\ \specialrule{0.1em}{0pt}{0pt}
			\endfirsthead
			
			\multicolumn{8}{c}%
			{{\bfseries \tablename\ \thetable{} -- continued from previous page}} \\
			\specialrule{0.1em}{0pt}{0pt}
			\multicolumn{1}{c}{No.} & 
			\multicolumn{1}{l}{Data set} & 
			\multicolumn{1}{l}{Method} &
			\multicolumn{1}{c}{($k,iter$)} &
			\multicolumn{1}{c}{$z(\mu^*,v^*)$} &
			\multicolumn{1}{c}{$\|\hat{\mathcal{E}}\|_2$} & 
			\multicolumn{1}{c}{$\|\epsilon\|_2$} &
			\multicolumn{1}{c}{$\|\mathcal{E}\|_2$}\\ \specialrule{0.1em}{0pt}{0pt}
			\endhead
			
			\hline \multicolumn{8}{r}{{Continued on next page}} \\
			\endfoot
			
			\hline \hline
			\endlastfoot			
			
			1	&	fourclass	&	imSN	&	$\left(	24	,	119	\right)$	&	0.098 	&	0.060 	&	0.023 	&	0.056 	\\	
			&		&	exSN	&	$\left(	24	,	119	\right)$	&	0.098 	&	0.060 	&	0.023 	&	0.056 	\\	\hline
			2	&	diabetes	&	imSN	&	$\left(	14	,	153	\right)$	&	0.082 	&	0.096 	&	0.041 	&	0.087 	\\	
			&		&	exSN	&	$\left(	14	,	153	\right)$	&	0.082 	&	0.096 	&	0.041 	&	0.087 	\\	\hline
			3	&	breast-cancer	&	imSN	&	$\left(	17	,	176	\right)$	&	0.092 	&	0.084 	&	0.039 	&	0.074 	\\	
			&		&	exSN	&	$\left(	17	,	176	\right)$	&	0.092 	&	0.084 	&	0.039 	&	0.074 	\\	\hline
			4	&	heart	&	imSN	&	$\left(	13	,	108	\right)$	&	0.187 	&	0.098 	&	0.053 	&	0.082 	\\	
			&		&	exSN	&	$\left(	13	,	108	\right)$	&	0.187 	&	0.098 	&	0.053 	&	0.082 	\\	\hline
			5	&	australian	&	imSN	&	$\left(	93	,	1927	\right)$	&	0.163 	&	0.090 	&	0.036 	&	0.083 	\\	
			&		&	exSN	&	$\left(	34	,	552	\right)$	&	0.084 	&	0.088 	&	0.030 	&	0.082 	\\	\hline
			6	&	svmguide4	&	imSN	&	$\left(	9	,	30	\right)$	&	0.103 	&	0.078 	&	0.061 	&	0.049 	\\	
			&		&	exSN	&	$\left(	9	,	30	\right)$	&	0.103 	&	0.078 	&	0.061 	&	0.049 	\\	\hline
			7	&	german.number	&	imSN	&	$\left(	13	,	180	\right)$	&	0.303 	&	0.095 	&	0.051 	&	0.080 	\\	
			&		&	exSN	&	$\left(	14	,	206	\right)$	&	0.209 	&	0.078 	&	0.040 	&	0.067 	\\	\hline
			8	&	ionosphere	&	imSN	&	$\left(	12	,	150	\right)$	&	3.160 	&	0.071 	&	0.043 	&	0.057 	\\	
			&		&	exSN	&	$\left(	12	,	150	\right)$	&	3.160 	&	0.071 	&	0.043 	&	0.057 	\\	\hline
			9	&	sonar	&	imSN	&	$\left(	12	,	202	\right)$	&	0.300 	&	0.091 	&	0.056 	&	0.072 	\\	
			&		&	exSN	&	$\left(	12	,	204	\right)$	&	0.300 	&	0.091 	&	0.056 	&	0.072 	\\	\hline
			10	&	phishing	&	imSN	&	$\left(	23	,	300	\right)$	&	0.067 	&	0.095 	&	0.031 	&	0.090 	\\	
			&		&	exSN	&	$\left(	23	,	300	\right)$	&	0.067 	&	0.095 	&	0.031 	&	0.090 	\\	\hline
			11	&	a2a	&	imSN	&	$\left(	21	,	472	\right)$	&	0.304 	&	0.072 	&	0.032 	&	0.065 	\\	
			&		&	exSN	&	$\left(	20	,	421	\right)$	&	0.202 	&	0.070 	&	0.031 	&	0.063 	\\	\hline
			12	&	a3a	&	imSN	&	$\left(	32	,	612	\right)$	&	0.131 	&	0.069 	&	0.028 	&	0.063 	\\	
			&		&	exSN	&	$\left(	21	,	423	\right)$	&	0.223 	&	0.078 	&	0.031 	&	0.072 	\\	\hline
			13	&	a4a	&	imSN	&	$\left(	20	,	405	\right)$	&	0.148 	&	0.073 	&	0.030 	&	0.067 	\\	
			&		&	exSN	&	$\left(	21	,	445	\right)$	&	0.144 	&	0.089 	&	0.026 	&	0.085 	\\	\hline
			14	&	a6a	&	imSN	&	$\left(	27	,	658	\right)$	&	0.248 	&	0.079 	&	0.030 	&	0.073 	\\	
			&		&	exSN	&	$\left(	19	,	436	\right)$	&	0.136 	&	0.091 	&	0.043 	&	0.080 	\\	\hline
			15	&	a7a	&	imSN	&	$\left(	50	,	1186	\right)$	&	0.119 	&	0.067 	&	0.025 	&	0.062 	\\	
			&		&	exSN	&	$\left(	21	,	511	\right)$	&	0.174 	&	0.087 	&	0.039 	&	0.078 	\\	\hline
			16	&	a9a	&	imSN	&	$\left(	23	,	504	\right)$	&	0.274 	&	0.081 	&	0.034 	&	0.074 	\\	
			&		&	exSN	&	$\left(	45	,	1052	\right)$	&	0.250 	&	0.065 	&	0.025 	&	0.060 	\\	\hline
			17	&	w1a	&	imSN	&	$\left(	23	,	196	\right)$	&	0.072 	&	0.078 	&	0.030 	&	0.072 	\\	
			&		&	exSN	&	$\left(	23	,	196	\right)$	&	0.072 	&	0.078 	&	0.030 	&	0.072 	\\	\hline
			18	&	w2a	&	imSN	&	$\left(	28	,	303	\right)$	&	0.106 	&	0.098 	&	0.031 	&	0.093 	\\	
			&		&	exSN	&	$\left(	28	,	303	\right)$	&	0.106 	&	0.098 	&	0.031 	&	0.093 	\\	\hline
			19	&	w4a	&	imSN	&	$\left(	82	,	897	\right)$	&	0.097 	&	0.076 	&	0.031 	&	0.070 	\\	
			&		&	exSN	&	$\left(	51	,	575	\right)$	&	0.106 	&	0.087 	&	0.037 	&	0.078 	\\	\hline
			20	&	w5a	&	imSN	&	$\left(	24	,	265	\right)$	&	0.101 	&	0.073 	&	0.030 	&	0.066 	\\	
			&		&	exSN	&	$\left(	24	,	265	\right)$	&	0.101 	&	0.073 	&	0.030 	&	0.066 	\\	\hline	
		\end{longtable}
	\end{center}
	
	\section{Conclusion}\label{sec6}
	In this paper, we proposed a bilevel optimization model for the hyperparameter selection for SVC in which both the upper-level problem and the lower-level problem are based on the logistic loss. 
	We reformulated the bilevel optimization problem into a single-level NLP based on the KKT condition of lower-level problem. Such nonlinear programming contains a set of nonlinear equality constraints and a simple lower bound constraints.
	To solve such NLP, we applied the smoothing Newton method proposed in \cite{Liang} to solve the KKT system, which contains one pair of complementarity constraints. It is proven that the algorithm has a superlinear convergence rate.
	Extensive numerical results on the data sets from the LIBSVM library verified the efficiency of the proposed approach (SN) over almost all the data sets used in this paper, which can achieve competitive results while consuming less time than other methods, and strict local minimizers can be achieved both numerically and theoretically. 
	The proposed approach has the potential to deal with other hyperparameter classification problems in SVM, which may involve multiple hyperparameters or multiple classes. These topics will be investigated further in the near future.

	\vskip 1cm
	Disclosure statement:
	There are no relevant financial or non-financial competing interests to report.


\begin{thebibliography}{99}\label{Reference}
		
		\bibitem{Abdalrada}
		Abdalrada A. S., Yahya O. H., Alaidi A. H. M., Hussein N. A., Alrikabi H. T., Al-Quraishi T. A. Q. (2019). A predictive model for liver disease progression based on logistic regression algorithm. Periodicals of Engineering and Natural Sciences, 7(3), 1255-1264.
		
		\bibitem{Barrett}
		Barrett R., Berry M., Chan T. F., Demmel J., Donato J., Dongarra J., ..., Van der Vorst, H. (1994). Templates for the solution of linear systems: building blocks for iterative methods. Society for Industrial and Applied Mathematics, 1-18
		
		\bibitem{Bennett}
		Bennett K. P., Hu J., Ji X., Kunapuli G., Pang J. S. (2006). Model selection via bilevel optimization. The 2006 IEEE International Joint Conference on Neural Network Proceedings, 1922-1929. 
		
		\bibitem{Chapelle}
		Chapelle O., Vapnik V., Bousquet O., Mukherjee S. (2002). Choosing multiple parameters for support vector machines. Machine Learning, 46, 131-159.
		
		\bibitem{Chaudhuri}
		Chaudhuri K., Monteleoni C., Sarwate A. D. (2011). Differentially private empirical risk minimization. Journal of Machine Learning Research, 12(3). 1069-1109
		
		\bibitem{Chauhan}
		Chauhan V. K., Dahiya K., Sharma A. (2019). Problem formulations and solvers in linear SVM: a review. Artificial Intelligence Review, 52(2), 803-855.
		
		\bibitem{Chen}
		Chen B., Harker P. T. (1993). A non-interior-point continuation method for linear complementarity problems. SIAM Journal on Matrix Analysis and Applications, 14(4), 1168-1190.
		
		\bibitem{Cokluk}
		Cokluk, O. (2010). Logistic Regression: Concept and Application. Educational Sciences: Theory and Practice, 10(3), 1397-1407.
		
		\bibitem{Colson}
		Colson B., Marcotte P., Savard G. (2007). An overview of bilevel optimization. Annals of Operations Research, 153, 235-256.
		
		\bibitem{Coniglio}
		Coniglio S., Dunn A., Li Q., Zemkoho A. (August 2023). Bilevel hyperparameter optimization for nonlinear support vector machines, manuscript
		
		\bibitem{Cortes}
		Cortes C., Vapnik, V. (1995). Support-vector networks. Machine Learning, 20, 273-297.
		
		\bibitem{Couellan}
		Couellan N., Wang W. (2015). Bi-level stochastic gradient for large scale support vector machine. Neurocomputing, 153, 300-308.
		
		\bibitem{Dayton}
		Dayton C. M. (1992). Logistic regression analysis. Stat, 474, 574.
		
		\bibitem{Dong}
		Dong Y. L., Xia Z. Q., Wang M. Z. (2007). An MPEC model for selecting optimal parameter in support vector machines. In The First International Symposium on Optimization and Systems Biology, 351-357.
		
		\bibitem{Duan}
		Duan K., Keerthi S. S., Poo A. N. (2003). Evaluation of simple performance measures for tuning SVM hyperparameters. Neurocomputing, 51, 41-59.
		
		\bibitem{Franceschini}
		Franceschini, F., Rafele, C. (2000). Quality evaluation in logistic services. International Journal of Agile Management Systems, 2(1), 49-54.
		
		\bibitem{Gill}
		Gill P E, Murray W, Saunders M A. (2005). SNOPT: An SQP algorithm for large-scale constrained optimization. SIAM Review, 47(1), 99-131.
		
		\bibitem{Glonek}
		Glonek G. F., McCullagh P. (1995). Multivariate logistic models. Journal of the Royal Statistical Society: Series B (Methodological), 57(3), 533-546.
		
		\bibitem{Kanzow}
		Kanzow, C. (1996). Some noninterior continuation methods for linear complementarity problems. SIAM Journal on Matrix Analysis and Applications, 17(4), 851-868.
		
		\bibitem{Keerthi}
		Keerthi S., Sindhwani V., Chapelle O. (2006). An efficient method for gradient-based adaptation of hyperparameters in SVM models. Advances in Neural Information Processing Systems, 19.
		
		\bibitem{Kleinbaum}
		Kleinbaum D. G., Klein M., Pryor E. R. (2002). Logistic regression: a self-learning text (Vol. 94). New York: springer.
		
		\bibitem{Kunapuli}
		Kunapuli G. (2008). A bilevel optimization approach to machine learning. Rensselaer Polytechnic Institute, NewYork
		
		\bibitem{Kunapuli_a}
		Kunapuli G., Bennett K. P., Hu J., Pang J. S. (2008). Classification model selection via bilevel programming. Optimization Methods \& Software, 23(4), 475-489.
		
		\bibitem{Kunapuli_b}
		Kunapuli G., Bennett K., Hu J., Pang J. S. (2008). Bilevel model selection for support vector machines. In CRM Proceedings and Lecture Notes, 45, 129-158.
		
		\bibitem{Kunisch}
		Kunisch K., Pock T. (2013). A bilevel optimization approach for parameter learning in variational models. SIAM Journal on Imaging Sciences, 6(2), 938-983.
		
		\bibitem{LaValley}
		LaValley M. P. (2008). Logistic regression. Circulation, 117(18), 2395-2399.
		
		\bibitem{Le}
		Le T., Vo B., Fujita H., Nguyen N. T., Baik S. W. (2019). A fast and accurate approach for bankruptcy forecasting using squared logistics loss with GPU-based extreme gradient boosting. Information Sciences, 494, 294-310.
		
		\bibitem{LiT}
		Li T., Beirami A., Sanjabi M., Smith, V. (2020). Tilted Empirical Risk Minimization. In International Conference on Learning Representations, 1-44.
		
		\bibitem{LiQ}
		Li Q., Li Z., Zemkoho A. (2022). Bilevel hyperparameter optimization for support vector classification: theoretical analysis and a solution method. Mathematical Methods of Operations Research, 1-36.
		
		\bibitem{LiQ_b}
		Li Q., Qian Y., Zemkoho A. (August 2023). LP-newton based global relaxation method for bilevel hyperparameter selection in support vector machines, manuscript.
		
		\bibitem{LiZ}
		Li Z., Qian Y., Li Q. (2022). A Unified Framework and a Case Study for Hyperparameter Selection in Machine Learning via Bilevel Optimization, 2022 5th International Conference on Data Science and Information Technology (DSIT), 1-8.
		
		\bibitem{Liang}
		Liang L., Sun D., Toh K. C. (2023). A squared smoothing Newton method for semidefinite programming. arXiv preprint arXiv:2303.05825.
		
		\bibitem{Meyer}
		Meyer P. (1994). Bi-logistic growth. Technological forecasting and social change, 47(1), 89-102.
		
		\bibitem{Momma}
		Momma M., Bennett K. P. (2002). A pattern search method for model selection of support vector regression. In Proceedings of the 2002 SIAM International Conference on Data Mining. Society for Industrial and Applied Mathematics, 261-274.
		
		\bibitem{Moore}
		Moore G. M., Bergeron C., Bennett K. P. (2009). Nonsmooth bilevel programming for hyperparameter selection. 2009 IEEE International Conference on Data Mining Workshops, 374-381.
		
		\bibitem{Nelder}
		Nelder J. A., Wedderburn R. W. (1972). Generalized linear models. Journal of the Royal Statistical Society: Series A (General), 135(3), 370-384.
		
		\bibitem{Nguyen}
		Nguyen T., Sanner S. (2013). Algorithms for direct 0–1 loss optimization in binary classification. In International conference on machine learning. PMLR, 1085-1093.
		
		\bibitem{Nocedal}
		Nocedal J., Wright S. J. (Eds.). (1999). Numerical optimization. New York, NY: Springer New York.
		
		\bibitem{Okuno}
		Okuno T., Takeda A., Kawana A., Watanabe M. (2021). On $l_p$-hyperparameter learning via bilevel nonsmooth optimization. The Journal of Machine Learning Research, 22(1), 11093-11139.
		
		\bibitem{Peduzzi}
		Peduzzi P., Concato J., Kemper E., Holford T. R., Feinstein A. R. (1996). A simulation study of the number of events per variable in logistic regression analysis. Journal of Clinical Epidemiology, 49(12), 1373-1379.
		
		\bibitem{Pinar}
		Pinar M. Ç., Zenios, S. A. (1994). On smoothing exact penalty functions for convex constrained optimization. SIAM Journal on Optimization, 4(3), 486-511.
		
		\bibitem{Pregibon}
		Pregibon D. (1981). Logistic regression diagnostics. The Annals of Statistics, 9(4), 705-724.
		
		\bibitem{Qi}
		Qi L., Sun D., Zhou G. (2000). A new look at smoothing Newton methods for nonlinear complementarity problems and box constrained variational inequalities. Mathematical programming, 87, 1-35.
		
		\bibitem{Qi_b}
		Qi, L., Sun, J. (1993). A nonsmooth version of Newton's method. Mathematical programming, 58(1-3), 353-367.
		
		\bibitem{Salehi}
		Salehi F., Abbasi E., Hassibi B. (2019). The impact of regularization on high-dimensional logistic regression. Advances in Neural Information Processing Systems, 32.
		
		\bibitem{Smale}
		Smale S. (2000). Algorithms for solving equations. In The Collected Papers of Stephen Smale, 3, 1263-1286.
		
		\bibitem{Stoltzfus}
		Stoltzfus, J. C. (2011). Logistic regression: a brief primer. Academic Emergency Medicine, 18(10), 1099-1104.
		
		\bibitem{Sun}
		Sun D., Qi L. (1999). On NCP-functions. Computational Optimization and Applications, 13, 201-220.
		
		\bibitem{Tripepi}
		Tripepi G., Jager K. J., Dekker F. W., Zoccali C. (2008). Linear and Logistic Regression Analysis. Kidney International, 73(7), 806-810.
		
		\bibitem{Vapnik}
		Vapnik, V. (2013). The nature of statistical learning theory: Springer Science \& Business Media. Berlin, Germany.
		
		\bibitem{Walkling}
		Walkling R. A. (1985). Predicting tender offer success: A logistic analysis. Journal of Financial and Quantitative Analysis, 20(4), 461-478.
		
	\end{thebibliography}
\end{document}